\documentclass[reqno,final,11pt]{amsart}
\usepackage{natbib}  
\usepackage{fancyhdr} 
\usepackage{color} 
\usepackage{hyperref} 
\usepackage{graphicx} 
\usepackage{geometry}
\usepackage[utf8]{inputenc}
\usepackage[T1]{fontenc}
\usepackage{verbatim}

\usepackage{amsmath}
\usepackage{amsthm}
\usepackage{amsfonts}
\usepackage{amssymb}
\usepackage{dsfont}
\usepackage{bm}
\usepackage{parskip}

\usepackage[scr=boondox]{mathalfa}

\usepackage{pgf,tikz}
\usetikzlibrary{arrows,shapes,positioning}
\usetikzlibrary{arrows.meta}
\usetikzlibrary{decorations.pathmorphing}
\usetikzlibrary{patterns,decorations.pathreplacing}
\usetikzlibrary{patterns}

\definecolor{aleacolor}{rgb}{0.16,0.59,0.78}

\hypersetup{
breaklinks,
colorlinks=true,
linkcolor=aleacolor,
urlcolor=aleacolor,
citecolor=aleacolor}

\renewcommand{\cite}{\citet}

\theoremstyle{plain}
\newtheorem{theorem}{Theorem}[section]                                          
\newtheorem{proposition}[theorem]{Proposition}                          
\newtheorem{lemma}[theorem]{Lemma}
\newtheorem{corollary}[theorem]{Corollary}

\theoremstyle{definition}
\newtheorem{definition}[theorem]{Definition}
\theoremstyle{remark}
\newtheorem{remark}[theorem]{Remark}

\makeatletter \@addtoreset{equation}{section} \makeatother
\renewcommand\theequation{\thesection.\arabic{equation}}
\renewcommand\thefigure{\thesection.\arabic{figure}}
\renewcommand\thetable{\thesection.\arabic{table}}

\newcommand{\R}{\mathds{R}}
\newcommand{\N}{\mathds{N}}

\newcommand{\ind}{\mathds{1}}
\newcommand{\E}[1]{\mathds{E}\left[#1\right]}
\newcommand{\itg}{\displaystyle\int}

\newcommand{\M}{\bm{M}}
\newcommand{\Ma}{\mathcal{M}}
\newcommand{\Ge}[2]{\bm{\Gel(} #1,#2 \bm{)} }
\newcommand{\T}{\mathcal{T}}
\newcommand{\A}{\mathcal{A}}

\newcommand{\C}{\mathcal{C}}

\newcommand{\Z}{\mathds{Z}}
\renewcommand{\L}{\bm{\mathscr{L}}}
\newcommand{\ones}{\bm{1}}
\newcommand{\e}{\bm{e}}
\renewcommand{\S}{\mathbb{S}}
\newcommand{\F}{\mathcal{F}}
\newcommand{\h}{\bm{h}}

\newcommand{\Ta}{\mathscr{T}}
\newcommand{\bTa}[1]{\bm{\mathscr{T}}_{\mkern-7mu #1}}
\newcommand{\g}{\bm{g}}
\renewcommand{\H}{\bm{H}}
\renewcommand{\Re}{\bm{\mathscr{R}}}
\newcommand{\Vone}[1]{\Vert #1 \Vert_{L^1}}
\newcommand{\Vity}[1]{\Vert #1 \Vert_{L^{\infty}}}
\newcommand{\Vityty}[1]{\Vert #1 \Vert_{L^{\infty}L^{\infty}}}
\newcommand{\Vityone}[1]{\Vert #1 \Vert_{L^{\infty}L^{1}}}
\newcommand{\pos}{\R_+^{\M}}
\newcommand{\U}{\mathcal{U}}

\newcommand{\Pois}{\mathscr{P}\!\mathscr{o}\!\mathscr{i}\!\mathscr{s}}

\newcommand{\Vvert}{\vert \hspace{-1pt} \vert \hspace{-1pt} \vert}

\newcommand{\CE}{\ce_{0}}
\newcommand{\CEt}{\ce}

\newcommand{\Res}{\mathcal{R}}

\newcommand{\trp}{^{\intercal}}

\newcommand{\fonction}[5]{\begin{array}[t]{lrcl}
#1: & #2 & \longrightarrow & #3 \\
    & #4 & \longmapsto & #5 \end{array}}

\DeclareMathOperator{\card}{Card}
\DeclareMathOperator{\diag}{diag}
\DeclareMathOperator{\spr}{SpR}
\DeclareMathOperator{\Gel}{Ge}
\DeclareMathOperator{\Id}{Id}
\DeclareMathOperator{\tp}{tp}
\DeclareMathOperator{\bd}{bd}
\DeclareMathOperator{\ce}{LC}
\DeclareMathOperator{\Lip}{Lip}

\begin{document}

\title[Exponential Tail Estimates for Poisson Branching Processes]{Exponential Tail Estimates for Multitype Poisson Branching Processes and Application to Hawkes Processes}

\author{Théo Leblanc}
\address{CEREMADE, Université Paris Dauphine - PSL}
\email{leblanc@ceremade.dauphine.fr} 

\thanks{}
\subjclass[2020]{60J80, 60E15, 60G55, 60J85, 47J26} 
\keywords{Bienaymé-Galton-Watson tree, Poisson Cluster, Exponential moments, Hawkes process, Multitype, Branching process, Fixed-point equations}

\begin{abstract}
	We establish exponential moment bounds for linear functionals of the total progeny of multitype Poisson Bienaymé-Galton-Watson trees. Our estimates are explicitly characterized in terms of the offspring mean matrix and the coefficients of the linear functional. As an application, we derive type-specific exponential moment bounds for multitype Hawkes processes, yielding improved results in high-dimensional settings. We also obtain exponential tail estimates for inhomogeneous Poisson clusters, with bounds that reflect the decay properties of the interaction kernels defining the clusters. These results provide useful probabilistic tools for the analysis of branching structures arising in Hawkes processes and related models.
\end{abstract}

\maketitle

\section{Introduction}

Bienaymé-Galton-Watson (BGW) trees are among the simplest classes of branching processes and serve as fundamental building blocks for many more complex stochastic models. They describe the evolution of a population originating from a single ancestor. The key property of BGW trees is that individuals reproduce independently of one another, each producing a random number of offspring according to common offspring distributions; see, for instance, \cite{athreya_branching_1972,harris_theory_1963}.

Multitype Bienaymé-Galton-Watson trees extend this framework by allowing individuals to belong to different types or classes. In this setting, the offspring distribution depends on the type of the parent, and offspring may themselves belong to different types. As in the single-type case, reproduction occurs independently across individuals, but the multitype structure enables the modeling of heterogeneous populations and interactions between different categories of individuals.

In the present article, we focus on the case where the offspring distributions are Poisson and characterized by a matrix $\H$, in the sense that an individual of type $i$ produces $\Pois(H_{ij})$ offsprings of type $j$. This framework remains quite general and naturally arises in the study of Hawkes processes, which have important applications in biology, epidemiology, finance, neuroscience, and earthquake modeling, among many others; see for example \cite{hawkes_spectra_1971,ogata_statistical_1988,lambert_reconstructing_2018,bacry_limit_2013}. The connection between Poisson BGW trees and Hawkes processes is provided by Poisson cluster representations. Indeed, Hawkes processes admit a cluster representation in which clusters arrive according to a Poisson process and each cluster is distributed as a Poisson BGW tree endowed with temporal marks \cite{hawkes_cluster_1974}. More precisely, an individual of type $i$ born at time $s$ produces offsprings of type $j$ according to a Poisson point process with intensity $h_{ij}(t-s)\,dt$, where $\h=(h_{ij})_{i,j\in\M}$ are called interaction functions.

The branching structure of BGW trees leads naturally to fixed-point equations satisfied by their Laplace transforms; see for example \cite{dwass_total_1969,good_number_1949,harris_theory_1963}. In the single-type setting, these equations are well understood through their connection with hitting times of random walks \cite{dwass_total_1969,good_number_1949,wang_total_2014}. When the offspring distribution is Poisson, the distribution of the total progeny is even explicit and is given by the Borel distribution; see \cite{bordenave_large_2007,bondesson_class_2004}. In contrast, the multitype case is substantially more involved, and explicit formulas are no longer available even in the Poisson setting; see \cite{karim_compound_2025,karim_exact_2021}.

For Hawkes processes, exponential estimates are classically derived through the cluster representation by controlling the contribution of each cluster. Previous works, such as \cite{reynaud-bouret_non_2007} and \cite{hansen_lasso_2015}, imposed decay assumptions on the interaction functions in order to guarantee the convergence of these contributions. However, \cite{leblanc_exponential_2024} showed that this convergence actually depends only on the matrix $\H$ which gathers the $L^1$ norms of the interaction functions, independently of their decay properties. More precisely, the relevant condition is simply that $\spr(\H)<1$, where $\spr$ denotes the spectral radius. This condition is also optimal, as it characterizes the existence of a stationary version of the process \cite{bremaud_stability_1996}.

A key result of \cite{leblanc_exponential_2024} is the derivation of quantitative exponential moment bounds for the total progeny of multitype BGW trees with offspring distributions $\Pois(\H)$. Nevertheless, these bounds become less informative in high-dimensional settings because they are not type-specific. In many applications, such as neuronal modeling where each type corresponds to a neuron, one is often interested only in the number of individuals associated with a single type, or with a small subset of types. By contrast, the bounds obtained in \cite{leblanc_exponential_2024} control the total population over all types simultaneously and therefore necessarily depend on the total number of types, which may itself become large.

Tail estimates for clusters also play a central role in the probabilistic analysis of Hawkes processes. In particular, they provide key tools for establishing properties such as mixing \cite{boly_mixing_2023} and regeneration \cite{graham_regenerative_2021}. This importance stems from the fact that the temporal length of clusters is intimately related to the dependence structure of Hawkes processes: long clusters generate long-range temporal dependencies, while quantitative control of cluster lengths yields quantitative control of dependence.

In this paper, we establish three main results. First, we derive quantitative exponential moment bounds for type-weighted total progeny of multitype Poisson BGW trees, together with applications to Hawkes processes. Second, we provide an exact characterization of the set of weights for which the exponential moment of the type-weighted total progeny is finite. Third, we obtain exponential tail estimates for type-weighted Poisson clusters in the general inhomogeneous setting.

\textbf{Exponential bounds for Poisson BGW trees.}
We establish a constructive characterization of exponential moments together with their functional properties and quantitative type-specific bounds; see Theorem \ref{theorem1}. Except in the single-type setting, where explicit formulas are available \cite{bordenave_large_2007}, no such type-specific quantitative bounds were previously known. In particular, \cite{leblanc_exponential_2024} provides only global bounds uniform over all types. We then apply these results to multitype Hawkes processes, obtaining type-specific exponential estimates for linear functionals of the process; see Theorem \ref{theorem2}. These estimates improve upon those of \cite{leblanc_exponential_2024}, especially in high-dimensional regimes. We also analyze how the bounds scale when one focuses on a single type or on a small subset of types, depending on the structure encoded by the matrix $\H$. These results are particularly relevant for statistical applications.

\textbf{Laplace domain of Poisson BGW trees.}
Beyond quantitative bounds on type-weighted total progeny, we characterize the set of weights for which the Laplace transform is finite. In the single-type setting, the critical weight is known explicitly; see \cite{bordenave_large_2007}. In \cite{karim_compound_2025}, Proposition 1, the authors derive an implicit characterization of both the Laplace transform and its domain in the multitype setting. Their framework allows arbitrary real weights, whereas we restrict attention to nonnegative weights. This restriction is natural for the statistical applications we have in mind and enables the use of monotonicity arguments, avoiding more abstract tools such as the preimage theorem employed in \cite{karim_compound_2025}. As a consequence, we obtain, in Theorem \ref{thm exact multD}, a more explicit characterization of the finiteness domain together with a clearer understanding of the multiple solutions of the associated fixed-point equation.

\textbf{Tail estimates for Poisson clusters.}
We derive type-specific estimates for the temporal tails of inhomogeneous Poisson clusters. Temporal tail bounds are obtained in three different regimes: exponentially decaying interaction functions, polynomially decaying interaction functions, and compactly supported interaction functions. This broad framework extends and improves earlier single-type estimates established mainly for compactly supported interactions in \cite{reynaud-bouret_non_2007}. The main originality of our work is that it goes beyond the homogeneous setting. Our results are established for inhomogeneous clusters, where the interaction functions may depend both on the parent birth time $s \in \R$ and the child birth time $t \in \R$, through functions $\h(s,t)$, rather than only on the delay $t-s$ through functions of the form $\h(t-s)$. Related inhomogeneous frameworks were considered in \cite{roueff_locally_2016}, where the authors studied regularity properties of Laplace transforms for locally stationary Hawkes processes in the single-type setting. The inhomogeneous framework substantially increases flexibility and enlarges the range of applications.

The paper is organized as follows. In Section \ref{section2}, we introduce the main objects and notation. Section \ref{section3} presents quantitative exponential bounds for type-weighted total progeny of multitype Poisson BGW trees and their applications to Hawkes processes. In Section \ref{section4}, we establish tail estimates for inhomogeneous Poisson cluster processes in the three decay regimes described above. Finally, Section \ref{section5} is devoted to the Laplace transform of the type-weighted total progeny of a Poisson BGW tree, including a characterization of its domain of finiteness and the construction of explicit values of the transform. Several technical results on bivariate convolutions used in Section \ref{section4} are collected in Appendix \ref{appendixA}.

\section{Settings}\label{section2}

Let $\N$ denote the set of nonnegative integers and let $\N^*=\N\setminus\{0\}$. Let $\R$ denote the set of real numbers and let $\R_+=[0,\infty)$ denote the set of nonnegative real numbers. Finally, let $M\in\N^*$ and denote by $\M=[\![1,M]\!]$ the set of types.

In this paper, trees are represented using the Ulam-Harris-Neveu encoding. Therefore, a tree $\tau$ is a non-empty subset of $\U:= \bigcup_{n\in\N} (\N^*)^n$ with $(\N^*)^0 = \{ \varnothing \}$ such that whenever $ak\in\tau$ for some $k\in\N^*$, one also has $a\in\tau$. In this case, we say that $b=ak$ is a child of $a$. It follows that $\varnothing \in \tau$, and we say that it is the root of $\tau$. For $a\in\tau$, by definition there exists $n\in\N$ such that $a\in (\N^*)^n$. The integer $n$ is called the generation of $a$ and is denoted $\vert a \vert$.

An $\M$-type tree, i.e. a tree with types in $\M$, is an object $(a,\tp(a))_{a\in\tau}$ where $\tau$ is a tree and $\tp(a)\in\M$ is the type of $a$. A random tree (resp. $\M$-type tree) is a random variable taking values in the set of trees (resp. $\M$-type trees).

\subsection{Multitype Poisson BGW trees}

In the sequel, when Poisson BGW trees are considered, the matrix $\H = (H_{ij})_{i,j\in \M}$ is a matrix with nonnegative entries, corresponding to the offspring mean matrix of the BGW tree. The following definition introduces $\M$-type BGW trees with offspring distributions $\Pois(\H)$.

\begin{definition}[$\Pois(\H)$-BGW tree]
	Let $m_0\in\M$. The random $\M$-type tree $\T^{m_0} = (a,\tp(a))_{a\in\A}$ is a $\Pois(\H)$-BGW tree with root of type $m_0$ if,
	\begin{itemize}
    	\item $\A$ is a random tree with root type $\tp(\varnothing)=m_0$.
    	\item Each individual $a\in\A$ with type $\tp(a)=m$ reproduces independently as follows: independently for each $m'\in\M$, it has $\Pois(H_{m,m'})$ children of type $m'$.
	\end{itemize}
\end{definition}

For a random $\M$-type tree $\T = (a,\tp(a))_{a\in\A}$ we define its type-count vector by
\begin{equation*}
    \card_{\M}(\T) = \big( \card(\{a\in\A \mid \tp(a)=m \})\big)_{m\in\M}.
\end{equation*}
For $n\in\N$ we define
\begin{equation*}
    \T_n = (a,\tp(a))_{a\in\A, \ \vert a \vert = n} \ , \quad \T_{\leq n} = (a,\tp(a))_{a\in\A, \ \vert a \vert \leq n} \ , \quad \T_{\geq n} = (a,\tp(a))_{a\in\A, \ \vert a \vert \geq n}.
\end{equation*}
The corresponding type-count vectors for $\T_n$, $\T_{\leq n}$, and $\T_{\geq n}$
are defined analogously.

\subsection{Inhomogeneous Poisson clusters}

Multitype Poisson BGW trees are closely linked to linear multivariate Hawkes processes and more generally to Poisson clusters. A Poisson cluster consists of a random Poisson $\M$-type tree together with additional temporal marks attached to each individual. These temporal marks are called birth dates.

When Poisson cluster processes are considered, their distribution is characterized by bivariate functions $\h = (h_{ij})_{i,j\in\M}$ where $h_{ij} : \R\times \R \longrightarrow \R_+$ are measurable functions. The matrix $\H$ is defined by $\H= \big(\sup_{s\in\R} \Vone{h_{ij}(s,\cdot)}\big)_{i,j\in\M}$, and is assumed to have finite entries, where $\Vone{\cdot}$ denotes the usual $L^1$ norm. 

\begin{definition}[$\Pois(\h)$-cluster]
	Let $m_0\in\M$ and $t_0\in\R$. Then, $G^{m_0}_{t_0}=(a,\tp(a),\bd(a))_{a\in\A}$, where $\A$ is a random tree, $\tp(a)\in\M$ is the type of $a$ and $\bd(a)\in\R$ is the birth date of $a$, is an inhomogeneous $\Pois(\h)$-cluster with root of type $m_0$ and born at time $t_0$ if,
\begin{itemize}
    \item The root of $\A$ has type $\tp(\varnothing)=m_0$ and birth date $\bd(\varnothing) = t_0$.
    \item Each $a\in\A$ with type $\tp(a)=m\in\M$ reproduces as follows:
    \begin{itemize}
        \item independently for each $m'\in\M$, $a$ has $\Pois(\Vone{ h_{m,m'}(\bd(a),\cdot)})$ children of type $m'$,
        \item independently, for each child $b$ of $a$ with type $\tp(b)=m'$ the birth date $\bd(b)$ has density
        \[ \bd(b) \sim \frac{h_{m,m'}(\bd(a),t)}{\Vone{h_{m,m'}(\bd(a),\cdot)}} dt.\]
    \end{itemize}
\end{itemize}
\end{definition}

\begin{remark}
	If, for every $i,j\in\M$, the function $h_{ij}(s,t)$ depends only on $t-s$, we are in the homogeneous case. With a slight abuse of notation, we still call $h_{ij}$ the function that satisfies $h_{ij}(s,t) = h_{ij}(t-s), \ s,t\in\R$. In this case $\Vone{h_{ij}(s,\cdot)}$ does not depend on $s$ and we have $\H = (\Vone{h_{ij}})_{ij\in\M}$. If in addition birth dates are removed, the underlying genealogical structure is a $\Pois(\H)$-BGW tree. If moreover we have $h_{ij}(s) = 0$ for $s \leq 0$, then $\Pois(\h)$-clusters corresponds to the usual clusters of Hawkes processes with interaction functions $\h$, see \cite{hawkes_spectra_1971} and \cite{hawkes_cluster_1974}.
\end{remark}
\begin{remark}
	We do not require $\h(s,t) =0$ for $t\leq s$ and thus, it is possible for children to be born before their parents.
\end{remark}

Given $I\subset \R$ and a cluster $G=(a,\tp(a),\bd(a))_{a\in\A}$ we define the trace of $G$ on $I$ by
\begin{equation*}
    G\cap I = (a,\tp(a),\bd(a))_{a\in\A, \ \bd(a)\in I} \ .
\end{equation*}

Finally we also define the type-count vector of $G\cap I$ by
\begin{equation*}
    \card_{\M}\big(G\cap I\big) = \Big(\card\big( \{ a\in\A \mid \bd(a)\in I, \ \tp(a)=m\}\big)\Big)_{m\in\M}.
\end{equation*}

\subsection{Notation}

Throughout this paper we use the following notation.

We denote by $\pos$ the set of vectors of $\R^{\M}$ with nonnegative entries,
\[\pos = \big\{x\in\R^{\M} \mid \forall i\in\M, \ x_i \geq  0\big\}.\]
For any vector $x\in\R^{\M}$ we denote $\vert x \vert_{\infty}$ the supremum norm on $\R^{\M}$ defined by $\vert x\vert_{\infty} = \max_{m\in\M} \vert x_m\vert$.\\
Similarly, for $p\ge1$, we define the usual $\ell^p$ norms on $\R^{\M}$ by, $\vert x\vert_p = \big(\sum_{m\in\M} \vert x_m\vert^p\big)^{1/p}$ for $p\geq 1$ and $x\in\R^{\M}$.\\
For any $1\leq p\leq \infty$, the operator norm associated to $\vert \cdot\vert_{p}$ is defined as follows. For $A$ a matrix of size $n\times q$ we have
\[\Vvert A \Vvert_{p} = \sup_{x\in\R^q\setminus \{0\}} \frac{\vert Ax\vert_{p}}{\vert x\vert_{p}}.\]
For $m\in\M$, the vector $\e_m\in\R^{\M}$ is the vector with zeros everywhere except at entry $m$ where it is one,
\[\e_m = (\ind_{m'=m})_{m'\in\M}.\] 
We denote by $\ones = (1,\ldots,1)\in\R^{\M}$ and $\Id$ the identity matrix of dimension $\M$.\\
The usual scalar product between two vectors $x = (x_m)_{m\in\M}\in\R^{\M}$ and $y=(y_m)_{m\in\M}\in\R^{\M}$ is
\[x \cdot y = x\trp y = \sum_{m\in\M} x_m y_m.\]
For any real function $f:\R\to\R$ and any vector $x\in\R^{\M}$, we define $f(x)$ by the entrywise application of $f$,
\[f(x)=\big(f(x_m)\big)_{m\in\M}.\]
For any vector $x$ we define $\diag(x)$ the diagonal square matrix, its dimension is given by the dimension of $x$, with diagonal entries given by $x$.\\
We denote by $\Ma_{n,p}(\R)$ the set of real matrices of size $n\times p$. For matrices $A,B\in\Ma_{n,p}(\R)$ we define the partial orders $\prec$ and $\preceq$ by:
\[ A \prec B \ \Leftrightarrow \ \forall i,j, \ A_{ij} < B_{ij}\]
and 
\[A \preceq B \ \Leftrightarrow \ \forall i,j, \ A_{ij} \leq  B_{ij}.\]
The spectral radius of a real square matrix $A$ is denoted by $\spr(A)$ and defined as the largest magnitude of the eigenvalues of $A$. 

The $L^1$ norm of a real function $f:\R\longrightarrow\R$ is defined by
\[ \Vone{f} = \itg_{\R} \vert f(x)\vert dx.\]
The $L^{\infty}$ norm of $f$ is defined by
\[\Vity{f} = \sup_{x\in\R} \vert f(x)\vert.\]

Finally, let us introduce the following definition.
\begin{definition}\label{GE}
    Let $A$ be a real square matrix. We say that $A$ satisfies the property $\Ge{r}{K}$ with $r,K\geq 0$ and denote $A\in\Ge{r}{K}$ if
    \begin{equation*}
        \forall n\in\N^*, \ \Vvert A^n \Vvert_{\infty} \leq Kr^n.
    \end{equation*}
\end{definition}

\begin{remark}
	The well-known Gelfand Theorem states that for any square matrix $A$, for any norm $\Vert \cdot \Vert$, we have $\Vert A^n \Vert ^{1/n} \xrightarrow[n\to\infty]{} \spr(A)$. Therefore, for every $r>\spr(A)$, there exists $K<\infty$ such that $A\in \Ge{r}{K}$ and we have $\spr(\H)<1 \iff \exists  r <1, \ \exists K<\infty, \ \H\in \Ge{r}{K}$. Since we are interested in subcritical BGW trees, the well-known condition $\spr(\H)<1$ can be replaced by $\H\in\Ge{r}{K}$ for some $r<1$ and $K<\infty$.
\end{remark}

\section{Exponential estimates for Poisson BGW trees and application to Hawkes processes}\label{section3}

In this section, we state our main results on exponential estimates for BGW trees and their applications to Hawkes processes. The proofs extensively use fixed-point and monotonicity arguments.

\subsection{Type-specific exponential estimates for Poisson BGW trees}\label{sec3sub1}

Theorem \ref{theorem1} is the main result of this section. It provides a constructive characterization together with quantitative bounds for the exponential moments of the type-weighted total progeny $u\trp \card_{\M}(\T)$ of a $\Pois(\H)$-BGW tree $\T$, where $u\in\pos$ is a vector of nonnegative weights. The theorem identifies the logarithmic Laplace transform through a nonlinear fixed-point equation and establishes explicit sufficient conditions on $\H$ and $u$ ensuring its finiteness. These estimates form the basis for the applications to Hawkes processes developed later in Section \ref{section3.3}. The proof of Theorem \ref{theorem1} is given in Section \ref{section3.2}.

\begin{theorem}\label{theorem1}
    Let $\T^{m}$ be a $\Pois(\H)$-BGW tree with root of type $m\in\M$. For $u\in\pos$, define $\L(u)\in [0,\infty]^{\M}$ by,
    \begin{equation}\label{th1eq1}
    	\exp\big( \e_{m}\trp \L(u)\big) = \E{e^{u\trp \card_{\M}(\T^{m})}}, \quad \forall m\in\M.
    \end{equation}
    Then $\L(u) \in [0,\infty]^{\M}$ satisfies the following properties.
    \begin{enumerate}
        \item The following finiteness criterion holds:
        \begin{equation}\label{th1eq2}
        \vert \L(u) \vert_{\infty} < \infty \ \Longleftrightarrow \ \exists  x\in\pos, \ x = u+\H(e^x-\ones).
        \end{equation}
        \item If $\vert \L(u) \vert_{\infty} < \infty$ then we have
        \begin{equation}\label{th1eq3}
            \L(u) = u+\H(e^{\L(u)}-\ones)
        \end{equation}
        and $\L(u)$ is the smallest solution (for $\preceq$) of Equation \eqref{th1eq3} among all solutions in $\pos$.       	
        \item For any $u,v \in \pos$ and $0\leq s\leq 1$ we have
    	\begin{equation}\label{th1eq4}
    		\L(u)+\L(v) \preceq \L(u+v), \quad \text{and} \quad \L(su+(1-s)v) \preceq s\L(u) + (1-s)\L(v).
    	\end{equation}
    	\item For $u\in\pos$, if $\spr\big(\H \diag(e^{\L(u)})\big) <1$ then we have 
    	\begin{equation}\label{th1eq5}
    		\L(u) \preceq \Big(\Id - \H \diag(e^{\L(u)})\Big)^{-1} u.
    	\end{equation}
        \item If $\H\in\Ge{r}{K}$ with $r<1$ and $K<\infty$, and if we define
        \begin{equation}\label{th1eq6}
            t_0(r,K) = \frac{\log\big( \frac{1+r}{2r}\big)}{1+\frac{2K}{1-r}},
        \end{equation}
        for all $u\in\pos$ with $\vert u\vert_{\infty} \leq t_0(r,K)$ the following quantitative bounds hold:
        \begin{equation}\label{th1eq7}
            \L(u) \preceq \L(\vert u\vert_{\infty}\ones) \preceq \vert u\vert_{\infty} \big( 1+ \frac{2K}{1-r}\big) \ones \preceq  \log\big( \frac{1+r}{2r}\big) \ones ,
        \end{equation}
        \begin{equation}\label{th1eq8}
            \L(u) \preceq  \Big( \Id - \frac{1+r}{2r} \H\Big)^{-1} u.
        \end{equation}
    \end{enumerate}
\end{theorem}

\begin{remark}
    All information about the exponential moments is encoded in the fixed-point equation \eqref{th1eq3} and that $\L(u)$ is the smallest nonnegative solution. Equation \eqref{th1eq3} is well-known, see for example \cite{wang_total_2014, karim_compound_2025}. Points $(4)$ and $(5)$ provide quantitative type-specific bounds for type-weighted total progeny of a Poisson BGW trees. Finiteness condition, point $(1)$, and characterization as the smallest nonnegative solution, point $(2)$, follow from our proof strategy based on monotonicity, which differ from the strategy used in \cite{karim_compound_2025}. These monotonicity arguments are available because we restrict attention to nonnegative weights, as opposed to Proposition 1 of \cite{karim_compound_2025}.
\end{remark}

\begin{remark}\label{remark3.3}
	As explained in Theorem \ref{thm exact multD} of Section \ref{section5}, if $\vert \L(u)\vert_{\infty}<\infty$ and $\spr(\H)<1$, then $\spr\big(\H\diag(e^{\L(u)})\big)<1$ unless $u$ is a boundary point of the domain of $\L$, in the sense that $\vert \L(tu)\vert_{\infty}=\infty$ for all $t>1$. At the boundary, one has $\spr\big(\H\diag(e^{\L(u)})\big)=1$.
\end{remark}

\begin{remark}
	In \eqref{th1eq8}, since $\H\in\Ge{r}{K}$, we have $\spr(\H)\leq r$ and thus for $\vert u\vert_{\infty}\leq t_0(r,K)$, it follows that $\spr(\frac{1+r}{2r} H) \leq \frac{1+r}{2} <1$, which guarantees that the inverse in \eqref{th1eq8} is well-defined.
\end{remark}

\begin{remark}
	Choosing $u = t\e_{m_0}$ with $t\leq t_0(r,K)$, equation \eqref{th1eq8} shows that $\L(t\e_{m_0})$ is bounded by the $m_0$-th column of $t\big(\Id-\frac{1+r}{2r}\H\big)^{-1}$. Consequently, the exponential moment of the number of individuals of type $m_0$ in a $\Pois(\H)$-BGW tree with root of type $m$ is controlled by the $(m,m_0)$-th entry of this matrix. By contrast, the global estimate obtained by taking $u=t\ones$, which is the framework of Lemma 3.2 \cite{leblanc_exponential_2024}, involves the full row sum of $t\big(\Id-\frac{1+r}{2r}\H\big)^{-1}$, which may be substantially larger. More generally, for $S\subset \M$, if $\ones_{S} = (\ind_{m\in S})_{m\in\M}$, $\H\in\Ge{r}{K}$ with $r<1$, and if $t\leq t_0(r,K)$, we have,
	\begin{equation}
		\E{e^{t\ones_{S}\trp \card_{\M}(\T^m)}} \leq \exp\Big[t\sum_{m'\in S} \big(\Id-\frac{1+r}{2r}\H\big)^{-1}_{m,m'}\Big], \quad m\in\M. 
	\end{equation}
\end{remark}

A natural question is whether the condition $\spr(\H)<1$ is necessary for the existence of a nonzero vector $u\in\pos$ such that $\vert\L(u)\vert_{\infty}<\infty$. We refer the reader to Section \ref{section5} for a more detailed discussion. In general, the answer is negative, as simple counterexamples can be constructed. Indeed, consider
\[\H= \begin{pmatrix} \alpha & 0 \\ 0 & 0 \end{pmatrix},\]
with $\alpha\geq 1$. Then $\spr(\H)=\alpha\geq 1$, yet for $u=\e_2$ we have $\L(\e_2)=\e_2$, which is finite. On the other hand, $\L(\e_1)=(\infty,0)$, so the first coordinate is infinite. Thus, the condition $\spr(\H)<1$ is not necessary for the existence of a nontrivial direction $u$ for which $\L(u)$ is finite.

The above example relies on the fact that $\H$ is reducible. To rule out this phenomenon, one needs an additional structural assumption on $\H$, such as irreducibility. Recall that a square matrix $A = (A_{ij})_{i,j\in \M}$ with nonnegative entries is irreducible if the graph with vertices $\M$ and containing directed edges $(i \rightarrow j)$ whenever $A_{ij} > 0$ is strongly connected. Equivalently, $A$ is irreducible if the following property holds,
\[\exists N\geq1, \ (\Id+A)^N \succ 0.\]
The following lemma provides a useful lower bound on $\L$ and yields our first results concerning its domain.

\begin{lemma}\label{lemma3.5}
	Let $u\in\pos$, and $\T^m$ be a $\Pois(\H)$-BGW tree. Then we have,
	\begin{equation}\label{lemma3.5eq1}
		\Big( \E{u\trp \card_{\M}(\T^m)} \Big)_{m\in\M} = \sum_{n\geq 0} \H^n u \preceq e^{\L(u)} - \ones.
	\end{equation}
	In particular the following holds.
	\begin{enumerate}
		\item If $\H$ is irreducible and $\spr(\H) \geq 1$. Then for any nonzero $u\in\pos$, we have
	\[ \forall m\in\M, \ \L(u)_m = \infty.\]
		\item If there exists $u\succ 0$ such that $\vert \L(u) \vert_{\infty}<\infty $ then we have $\spr(\H)<1$.
	\end{enumerate}
\end{lemma}

Equality \eqref{lemma3.5eq1} is well-known and not specific to the Poisson framework, $(1)$ and $(2)$ are also well-known, standard textbook results, we nevertheless provide the proofs for completeness. Irreducibility conditions are classical when one is interested in extinction probabilities, see for example \cite{harris_theory_1963}, and are thus linked to condition for finiteness of the Laplace transform.

\begin{proof}[Proof of Lemma \ref{lemma3.5}]
Clearly, we have
\[ \E{u\trp\card_{\M}(\T^m_0)} = u\trp \e_m = (u)_m.\]
Then, for $n=1$ we have
\[ \E{u\trp\card_{\M}(\T^m_1)} = \sum_{m'}H_{m,m'} u_{m'} = (\H u)_m.\]
By conditioning on the previous generation, a classical induction shows that
\[\E{u\trp\card_{\M}(\T^m_n)} =(\H^n u)_m, \quad n\geq 0.\]
Thus we have $\E{u\trp\card_{\M}(\T^m)} = \big(\sum_{n\geq 0} \H^n u\big)_m$ which is the result. From $e^x \geq 1+x$ the lower bound on $e^{\L(u)}$ is also clear. Let us now assume that $\spr(\H)\geq 1$ and that $\H$ is irreducible. Let $u\succeq 0$ and $u\neq 0$. By irreducibility, there exists $N\in\N^*$ such that $\sum_{k=0}^{N-1} \H^k \succ 0$. Thus, it is clear that $\frac{1}{N}\sum_{k=0}^{N-1} \H^ku \succeq \varepsilon\ones \succ 0$ for some $\varepsilon > 0$. It follows that
\begin{equation*}
	e^{\L(u)}  \succeq \ones + \Big(\sum_{k=0}^{\infty} \H^k\Big)u \succeq  \Big(\sum_{k=0}^{\infty} \H^k\Big) \Big(\frac{1}{N}\sum_{k=0}^{N-1} \H^ku\Big) \succeq \varepsilon \Big(\sum_{k=0}^{\infty} \H^k\Big) \ones.
\end{equation*}
By the Perron-Frobenius Theorem (see Appendix \ref{appendixA.1}), $\rho = \spr(\H)\geq 1$ is an eigenvalue of $\H$ with an eigenvector $\kappa \succ 0$. Since we may assume $\ones \succeq \kappa$, we have
\begin{equation*}
	e^{\L(u)}  \succeq \ones + \varepsilon \sum_{k=0}^{\infty} \rho^k \kappa  = \infty\cdot \ones,
\end{equation*} 
which concludes point $(1)$ by taking the logarithm. Suppose now that there exists $u\succ 0$ such that $\vert \L(u)\vert_{\infty}<\infty$. Let $\kappa\succeq0$ be an eigenvector associated with $\rho=\spr(\H)$, whose existence follows from the weak Perron-Frobenius Theorem (see Appendix \ref{appendixA.1}). For $\varepsilon>0$ small enough we have $u\succeq \varepsilon \kappa$ and thus $\L(\varepsilon \kappa)$ is finite. Considering the bound
\[e^{\L(\varepsilon \kappa)} \succeq \sum_{n\geq 0} \H^n \varepsilon\kappa = \varepsilon \Big(\sum_{n\geq 0} \rho^n\Big) \kappa,\]
we must have $\rho<1$ since $\kappa \neq 0$, which is point $(2)$.
\end{proof}

Lemma \ref{lemma3.5} establishes early properties of links between $\H$ and the domain where $\L$ is finite. In the single-type case we can find explicitly the largest $u\in\R_+$ such that the Laplace transform is finite. This well-known result, see for example \cite{bordenave_large_2007,bondesson_class_2004}, is stated below in Theorem \ref{theorem3.6}.

\begin{theorem}\label{theorem3.6}
	Let $\T$ be a single-type $\Pois(\alpha)$-BGW tree with $0<\alpha<1$. Let 
	\begin{equation}
		u_c = \log\big(1/\alpha\big) -(1-\alpha).
	\end{equation}
	Then, for $u\geq 0$ we have
	\begin{equation}
		\E{e^{u \card(\T)}} < \infty \iff u \leq u_c,
	\end{equation}
	and we have
	\begin{equation}
		\E{e^{u_c \card(\T)}} = \frac{1}{\alpha}.
	\end{equation}
	Finally, for any $x \geq 0$ such that $u := x - \alpha (e^x-1) \geq 0$ the following holds,
	\begin{equation}\label{th3.6eq1}
		x = \L(u) \iff \alpha e^x \leq 1.
	\end{equation}
\end{theorem}

This result is of interest since it is not clear a priori that the exponential moment at $u_c = \sup\{ u\geq 0 \mid \L(u)<\infty\}$ is finite. Also, by \eqref{th3.6eq1}, for any $0\leq x \leq \log(1/\alpha)$ we have $\L\big(x-\alpha(e^x-1)\big) = x$, which, up to an inverse problem, completely solves the question of the exponential moments in the one dimensional case. By mimicking the single-type case, in Section \ref{section5}, we derive the multitype extension of Theorem \ref{theorem3.6} and perform a complete study of the topological properties of the domain where $\L$ is finite. Thus the proof of Theorem \ref{theorem3.6} is of interest since it may give intuition in the multitype case and is given below.

\begin{proof}[Proof of Theorem \ref{theorem3.6}]
For $u\geq 0$ consider the real function $\phi_u(x) = u + \alpha(e^x-1) - x$. We want to find for which $u$ this function has nonnegative zeros. $\phi_u$ is strictly convex, $\phi_u(0)=u\geq 0$, $\phi_u'(0) < 0$ and $\lim_{x\to\infty} \phi_u(x) = \infty$ thus if $\phi_u$ has zeros, they must be nonnegative. A quick study of $\phi_u$ shows that its minimum is reached at $x = \log(1/\alpha)$. Thus, if $\min \phi_u  > 0$ there are no zeros, if $\min \phi_u  = 0$ there is one zero and finally if $\min \phi_u  < 0$ there are two zeros. Since for a fixed $x$ the quantity $\phi_u(x)$ is increasing with $u$, it is clear that the largest $u$ such that there exists at least one zero is the one such that \[0 = \min \phi_u = \phi_u(\log(1/\alpha)) =  u + (1-\alpha) - \log(1/\alpha).\]
	Which gives $u = u_c = \log(1/\alpha) - (1-\alpha) > 0$ and $x_c = \L(u_c) = \log(1/\alpha)$.\\
	The last statement comes from the fact that $\L(u)$ is the smallest nonnegative zero. Indeed if $0$ is a global minimum of $\phi_u$ (ie $u=u_c$) then there is a unique zero, $\L(u)$, and necessarily we have $\phi_u'(\L(u)) = 0$. If there are two distinct zeros $x_1$ and $x_2$, by strict convexity we have $\phi_u'(x_1)<0$ and $\phi_u'(x_2) > 0$. Since $\L(u) = x_1$ we have the result.
\end{proof}

\subsection[Proof of Theorem 1]{Proof of Theorem \ref{theorem1}}\label{section3.2}

The proof of Theorem \ref{theorem1} is based on several intermediate results. We begin with the following definition.

\begin{definition}\label{definition2}
    Let $u\in\pos$ and define $\psi_{u} : x\in\R^M \mapsto u + \H(e^x-\ones)$. Then we define the increasing limit
    \[\L(u) = \lim\limits_{n\to\infty} \psi_u^n(0)\]
    where $\psi^n_u$ is the $n$-fold composition $\psi_u  \circ \cdots \circ \psi_{u}$.
\end{definition}

Let us check that $(\psi_u^n(0))_n$ is increasing. Indeed, $\psi_u$ is increasing with respect to $\preceq$, and since $0\preceq \psi_u(0)=u$, by induction we have $\psi^{n}_u(0) \preceq \psi^{n+1}_u(0)$ for any $n\in\N$. The connection between this definition of $\L(u)$ and the exponential moments will be established later. For the moment, we work only with this definition and prove the different results on $\L(u)$. 

\begin{proposition}\label{proptemp1}
    Suppose that there exists a finite vector $x\succeq0$ such that $x = \psi_{u}(x)$. Then $\L(u)$ is finite and it is the smallest solution (for $\preceq$) of the fixed-point equation $y = \psi_{u}(y)$ among the solutions with nonnegative coordinates. Conversely, if $\vert \L(u)\vert_{\infty} < \infty$ then $\L(u)$ is solution of $y = u+\H(e^y-\ones)$.\\
Moreover, for any $0\preceq u\preceq v$, we have $\L(u) \preceq \L(v)$, and for any $u,v\succeq 0$ we have $\L(u)+\L(v) \preceq \L(u+v)$.
\end{proposition}

\begin{proof}[Proof of Proposition \ref{proptemp1}]
    Since $\L(u)$ is an increasing limit and since $\psi^0_u(0) = 0$ we always have $\L(u)\succeq 0$. Let $x$ be a finite nonnegative fixed point of $\psi_u$. Since $x \succeq 0$, applying $\psi_{u}$ on both sides gives $x \succeq \psi_{u}(0)$. By induction, $x \succeq \psi_{u}^n(0)$ for all $n\geq 0$ and thus at the limit, $x \succeq \L(u)$ which proves that $\L(u)$ is finite, nonnegative, and smaller than any other nonnegative solution. Continuity of $\psi_u$ also implies that $\L(u)$ is a solution of the fixed-point equation.\\
    Reciprocally, if $\vert \L(u)\vert_{\infty} < \infty$ it is clear by continuity of $\psi_{u}$ that $\L(u)$ is a nonnegative fixed-point.
    
    For the monotonicity properties, since $\psi_{u}$ and $\psi_{v}$ are increasing functions with respect to $\preceq$ and since $\psi_{u} \preceq \psi_{v}$ for $u\preceq v$ we have by induction,
    \[ \psi^n_{u}(0) \preceq \psi^n_{v}(0), \quad n\in\N\]
    which leads to $\L(u) \preceq \L(v)$.
    Similarly, since $e^{x+y}-1 \geq e^x-1 + e^y-1$ for any $x,y\geq 0$ it is clear that 
    \[\psi_{u+v}(x+y) \succeq \psi_u(x) + \psi_v(y)\]
    for any $x,y\succeq 0$. Thus an induction leads to $\psi_{u+v}^n(0) \succeq \psi_u^n(0) + \psi_v^n(0)$ which concludes on $\L(u)+\L(v) \preceq \L(u+v)$.
\end{proof}

Definition \ref{definition2} and Proposition \ref{proptemp1} immediately imply that either there exists a finite vector $x\succeq0$ satisfying $x=u+\H(e^x-\ones)$, or $\vert\L(u)\vert_\infty=\infty$. Thus the following holds
\begin{equation}
        \vert \L(u)\vert_{\infty} < \infty \ \Longleftrightarrow \ \exists x\in\pos, \ x = u+\H(e^x-\ones),
\end{equation}
which is point $(1)$ of Theorem \ref{theorem1}. It is also clear from Proposition \ref{proptemp1} that point $(2)$ and the first part of point $(3)$ of Theorem \ref{theorem1} hold. The second part of point $(3)$ is postponed at the end of the proof.

Lemma 3.2 of \cite{leblanc_exponential_2024} proves exponential moment for $\Pois(\H)$-BGW trees uniformly over all types. The cornerstone of the proof of this lemma is the limit behaviour of the recursive sequence $(\psi^n_{t\ones}(0))_{n\geq 0}$ with $t\geq 0$. The proof yields the following result.

\begin{lemma}\label{lemma1}
    Suppose that $\H\in\Ge{r}{K}$ with $r<1$. Then for all $0\leq t\leq t_0(r,K)$ we have
    \begin{equation}
    	\L(t\ones) = \lim\limits_{n\to\infty} \psi^n_{t\ones}(0) \preceq t \big( 1+ \frac{2K}{1-r} \big) \ones,
    \end{equation} 
    with $t_0(r,K)$ defined in Theorem \ref{theorem1}.
\end{lemma}

From this result we can easily derive the rest of the proof. Indeed, by Proposition \ref{proptemp1}, Lemma \ref{lemma1}, for $u\in\pos$ with $\vert u\vert_{\infty} \leq t_0(r,K)$ we have,
\begin{equation}\label{demth1eq16}
    \L(u) \preceq \L(\vert u\vert_{\infty}\ones) \preceq \vert u\vert_{\infty} \big( 1+ \frac{2K}{1-r} \big) \ones \preceq \log\big(\frac{1+r}{2r}\big) \ones
\end{equation}
which is equation \eqref{th1eq7}. Let us continue with the following.

\begin{proposition}\label{proposition2}
    Let $u\in\pos$. Then we have (where both sides may possibly be infinite on some entries),
    \begin{equation}
        \L(u) \preceq \sum_{n\geq 0} \big[ \H\diag(e^{\L(u)}) \big]^n u.
    \end{equation}
    Thus if we suppose that $\spr\big(\H\diag(e^{\L(u)})\big)<1$, we have
    \begin{equation}
        \L(u) \preceq \big(\Id - \H \diag(e^{\L(u)}) \big)^{-1} u.
    \end{equation}
    In particular if $\H\in\Ge{r}{K}$ with $r<1$ and if $u\in\pos$ with $ \vert u\vert_{\infty} \leq t_0(r,K)$ then we have
    \begin{equation}
         \L(u)  \preceq  \big(\Id - \frac{1+r}{2r} \H\big)^{-1} u.
    \end{equation}
\end{proposition}

\begin{proof}[Proof of Proposition \ref{proposition2}]
    Let $u\succeq 0$. First, observe that for $n\geq 1$ we can rewrite $\psi_{u}^n(0)$ as follows,
    \begin{equation*}
        \psi_{u}^n(0) = \sum_{k=0}^{n-1} \big(\psi_{u}^{k+1}(0)-\psi_{u}^k(0)\big).
    \end{equation*}
    For $k=0$, we have $\psi_{u}^{k+1}(0)-\psi_{u}^k(0)= u = \big[\H\diag(e^{\L(u)})\big]^0  u$. Using the inequality $e^x-e^y \preceq \diag(e^x)(x-y)$, valid for all $0\preceq y\preceq x$, for $k\geq 1$, we have
    \begin{align*}
        \psi_{u}^{k+1}(0)-\psi_{u}^k(0) & = \H\big( e^{\psi_{u}^k(0)} - e^{\psi_{u}^{k-1}(0)}\big) \\
        & \preceq \H\diag(e^{\psi_u^k(0)}) (\psi_{u}^k(0) - \psi_{u}^{k-1}(0)) \\
        & \preceq  \H \diag(e^{\L(u)}) (\psi_{u}^k(0) - \psi_{u}^{k-1}(0)) \\
        & \preceq \big[ \H\diag(e^{\L(u)}) \big]^k  u
    \end{align*}
    where the last inequality follows by induction on $k$. Thus we have
    \begin{align*}
        \L(u) & \preceq \sum_{k=0}^{\infty} \big[ \H\diag(e^{\L(u)}) \big]^k  u.
    \end{align*}
    Under the assumption $\spr\big(\H\diag(e^{\L(u)})\big)<1$, the inverse is well-defined and we obtain
    \[\L(u) \preceq  \big(\Id - \H\diag(e^{\L(u)})\big)^{-1} u.\]
    For the last part, if $\H\in\Ge{r}{K}$ and $\vert u\vert_{\infty} \leq t_0(r,K)$ then by \eqref{demth1eq16}, $\L(u) \preceq \log\big(\frac{1+r}{2r}\big)\ones$ and so
    \[ \H\diag(e^{\L(u)}) \preceq \frac{1+r}{2r} \H.\]
    Since $\spr\big(\frac{1+r}{2r} \H \big)\leq \frac{1+r}{2}<1$ the second bound is clear.
\end{proof}

We can now finish the proof of Theorem \ref{theorem1}. Equations \eqref{th1eq5} and \eqref{th1eq8} follow directly from Proposition \ref{proposition2}. Let us now link the definition of $\L$ by Definition \ref{definition2} to the one by Theorem \ref{theorem1}.

Let $u\in\pos$ and $n$ a nonnegative integer. The branching property of BGW trees leads to the following equality in distribution;
\begin{equation}
    \card_{\M}(\T^{m_0}_{\leq n+1}) = \e_{m_0} + \sum_{m\in\M} \sum_{k=1}^{X_{m_0,m}} \card_{\M}(\T^{m,k}_{\leq n})
\end{equation}
where $X_{m_0,m}, \ \T^{m,k}_{\leq n}$ for $m\in\M, \ k\in\N$ are independent, $X_{m_0,m} \sim \Pois(H_{m_0,m})$ and $\T^{m,k}_{\leq n} \sim \T^{m}_{\leq n}$.

Thus if we denote $g^n_{m_0}(u) = \log \E{e^{u\trp \card_{\M}(\T^{m_0}_{\leq n})}}$ and $\bm{g}^n(u) = (g^n_m(u))_{m\in\M}$ we have
\begin{align*}
    g^{n+1}_{m_0}(u) & = \log \Big(  \E{e^{u\trp \e_{m_0} + \sum_{m\in\M} \sum_{k=1}^{X_{m_0,m}} u\trp \card_{\M}(\T^{m,k}_{\leq n})}} \Big)\\
    & = u \trp \e_{m_0} + \sum_{m\in\M} \log \Big(  \E{e^{\sum_{k=1}^{X_{m_0,m}} u\trp \card_{\M}(\T^{m,k}_{\leq n})}} \Big) \\
    & = u \trp \e_{m_0} + \sum_{m\in\M} \log \bigg(  \E{\prod_{k=1}^{X_{m_0,m}} \E{e^{u\trp \card_{\M}(\T^{m}_{\leq n})}} }\bigg) \\
    & = u \trp \e_{m_0} + \sum_{m\in\M} \log \Big(  \E{(e^{g^n_m(u)})^{X_{m_0,m}} }\Big) \\
    & = u \trp \e_{m_0} + \sum_{m\in\M}  H_{m_0,m} (e^{g^n_m(u)}-1),
\end{align*}
where we used the moment generating function of the Poisson distribution in the last step. It follows that $\bm{g}^{n+1}(u) = \psi_{u}(\bm{g}^n(u))$. Finally since $\bm{g}^{0}(u) = u = \psi_{u}(0)$, we have $\bm{g}^{n}(u) = \psi_{u}^{n+1}(0)$. Definition \ref{definition2} therefore implies 
\begin{equation*}
    \lim\limits_{n\to\infty} \bm{g}^n(u) = \L(u).
\end{equation*}
Then by the monotone convergence theorem, $e^{g^n_{m_0}(u)} = \E{e^{u \trp \card_{\M}(\T^{m_0}_{\leq n})}} \xrightarrow[n\to\infty]{} \E{e^{u \trp \card_{\M}(\T^{m_0})}}$, which concludes that the definitions of $\L$ by \eqref{th1eq1} and by Definition \ref{definition2} coincide. Now that the link between the exponential moments and $\L(u)$ has been shown, all the other results are already proved except the convexity property in point $(3)$. Let $u,v\in\pos$, $m\in\M$ and $0\leq s \leq 1$ then by Hölder inequality we have
\begin{equation*}
	e^{\L(su+(1-s)v)_m} = \E{e^{su\trp \card_{\M}(\T^m)}e^{(1-s)v\trp \card_{\M}(\T^m)}} \leq \E{e^{u\trp \card_{\M}(\T^m)}}^s \E{e^{v\trp \card_{\M}(\T^m)}}^{1-s},
\end{equation*}
which proves the convexity property in point $(3)$ and concludes the proof of Theorem \ref{theorem1}.

\subsection{Application to Hawkes processes}\label{section3.3}

We now state Theorem \ref{theorem2}, which provides type-specific exponential moment bounds for multivariate Hawkes processes.

Consider $\bm{\mu}:=(\mu_m)_{m\in\M} \in \R_+^{\M}$ and homogeneous interaction functions $\h = (h_{ij})_{i,j\in\M}$, meaning that,
\[
h_{ij}(s,t)=h_{ij}(t-s), \qquad i,j\in\M, \ s,t\in\R.
\]

Let $(\Omega,\F,\mathds{P})$ be a probability space endowed with a filtration $(\F_t)_{t\in\R}$. A Hawkes process with parameters $(\M,\bm{\mu},\bm{h})$ is a multivariate point process $\bm{N}:=(N^m)_{m\in\M}$, i.e. random locally finite subsets of $\R$, such that
\begin{itemize}
    \item $\bm{N}$ is adapted to $(\F_t)_{t\in\R}$,
    \item almost surely, for all distinct $m,m'\in\M$, the point processes $N^m$ and $N^{m'}$ have no common points,
    \item $\bm{N}$ admits a predictable intensity process $\bm{\lambda}:=(\lambda^m)_{m\in\M}$ satisfying, where $dN^m$ denotes the counting measure associated with $N^m$, the following equality
    \begin{equation}\label{eq hawkes}
        \lambda_t^m = \mu_m + \sum_{m'\in\M} \itg_{-\infty}^{t-} h_{m',m}(t-s) dN_s^{m'}, \qquad m\in\M,\ t\in\R.
    \end{equation}
\end{itemize}

Informally, $\lambda_t^m dt$ represents the conditional probability that $N^m$ has a point in $(t,t+dt]$, given $\F_t$. One can interpret $\mu_m$ as the immigration rate of points of type $m$, and $h_{m',m}(t-s)$ as the rate at time $t$ for new points of type $m$ generated by a point of type $m'$ at time $s$. More precisely, the well-known cluster representation, see \cite{hawkes_cluster_1974}, states that $\bm{N}$ admits a cluster representation in which points are realised as birth times of independent $\Pois(\h)$-clusters with roots arriving according to a multitype Poisson process on $\R$ with intensity $\bm{\mu}$.

It was shown in \cite{bremaud_stability_1996} that such a process exists and admits a stationary version whenever the spectral radius of the matrix $\bm{H}=(\Vone{ h_{m,m'}} )_{m,m'\in\M}$ satisfies $\spr(\bm{H})<1$. 

Combining Theorem \ref{theorem1} with the cluster representation and the arguments developed in \cite{leblanc_exponential_2024} yields the following result, Theorem \ref{theorem2} below. The proof is given in Section \ref{section3.4}

\begin{theorem}\label{theorem2}
	Let $\bm{N}$ be a Hawkes process with parameters $\M,\bm{\mu},\h$. Let $\L$ denote the function introduced in Theorem \ref{theorem1} with $\H = (\Vert h_{m,m'} \Vert_1)_{m,m'\in\M}$. Let $\bm{f} = (f^m)_{m\in\M}$ be nonnegative measurable functions, ie $f^m : \R \rightarrow \R_+$ for $m\in\M$. Define 
	\[ \bm{N}(\bm{f}) = \sum_{m\in\M} \sum_{x\in N^m} f^m(x).\]
	Then, for any $T>0$ we have
	\begin{equation}
		\E{e^{\bm{N}(\bm{f})}} \leq \exp\left( T \bm{\mu}\trp \big(e^{\L(u_T(\bm{f}))}-\ones\big) \right),
	\end{equation}
	where
	\[ u_T(\bm{f}) := \Big(\sup_{x\in\R} \sum_{n\in\Z} f^m(x+nT)\Big)_{m\in\M}.\]
	Moreover, if $\H \in \Ge{r}{K}$ with $r<1$, $K<\infty$ and if $\vert u_T(\bm{f}) \vert_{\infty} \leq t_0(r,K) = \frac{\log((1+r)/2r)}{1+2K/(1-r)}$, the following quantitative bounds hold, where the second one also requires that $\H\trp \in \Ge{r}{\tilde{K}}$ for $\tilde{K}<\infty$.
	\begin{align}
		\bm{\mu}\trp \big(e^{\L(u_T(\bm{f}))}-\ones\big) \leq \ &  e^{\vert u_T(\bm{f}) \vert_{\infty}(1+\frac{2K}{1-r})} \bm{\mu}\trp \big(\Id-\frac{1+r}{2r}\H\big)^{-1}u_T(\bm{f})\label{quant2}\\
		\leq \ &  \vert \bm{\mu}\vert_{\infty} e^{\vert u_T(\bm{f}) \vert_{\infty} (1+\frac{2K}{1-r})}  \big(1+\frac{2\tilde{K}}{1-r}\big) \vert u_T(\bm{f}) \vert_1.\label{quant3}
	\end{align}
\end{theorem}

\begin{remark}
	Let $u\in \pos$ and $L>0$. Taking $\bm{f} = (u_m \ind_{[0,L)})$, gives $\bm{N}(\bm{f}) = u\trp \bm{N}([0,L))$ where $\bm{N}([0,L)) = (N^m([0,L))_{m\in\M}$ is the number of points of each type in $[0,L)$. Thus one can apply Theorem \ref{theorem2} with $T=L$ and obtain (quantitative bounds also holds) 
	\begin{equation*}
		\E{e^{u\trp \bm{N}([0,L))}} \leq  \exp\left( L \bm{\mu} \trp \big( e^{\L(u)} -\ones \big)\right).
	\end{equation*}
    Theorem 1.4 of \cite{leblanc_exponential_2024} is recovered by taking $u=t\ones$ with $0\leq t \leq t_0(r,K)$. Theorem \ref{theorem2} allows to take, for example, $u = t\e_m$ which yields sharper type-specific estimates.
\end{remark}

\begin{remark}
	The bounds \eqref{quant2} and \eqref{quant3} correspond to successive simplifications of the estimate obtained from Theorem \ref{theorem1}. Equation \eqref{quant2} is the combination of the bound \eqref{th1eq8} applied to $\L(u_T(\bm{f}))$ and inequality $e^x-\ones \preceq e^{\vert x\vert_{\infty}} x, \ x\in\pos$. Since $e^{\vert u_T(\bm{f}) \vert_{\infty}(1+\frac{2K}{1-r})} \leq \frac{1+r}{2r}$, equation \eqref{quant2} becomes sublinear with respect to $u_T(\bm{f})$. Finally, \eqref{quant3} follows from the estimate $x\trp A y \leq \vert x\vert_{\infty} \Vvert A\Vvert_1 \vert y\vert_1$ for $x,y$ vectors and $A$ a matrix. Assumption $\H\trp \in \Ge{r}{\tilde{K}}$ naturally leads to bounds on the norm $\Vvert \big(\Id-\frac{1+r}{2r}\H\big)^{-1}\Vvert_1$ since $\Vvert A\trp \Vvert_{\infty} = \Vvert A \Vvert_{1}$ for any matrix $A$, and thus,
	\begin{equation*}
		\H^{\intercal} \in \Ge{r}{\tilde{K}} \quad \iff \quad \forall n\geq 1, \  \Vvert \H^n \Vvert_{1} \leq \tilde{K} r^n.
	\end{equation*}
\end{remark}

\begin{remark}
	The result also holds for nonlinear Hawkes processes with intensities
	\begin{equation*}
		\lambda^m_t = \Phi^m\Big(\mu_m + \sum_{m'\in\M}\itg_{-\infty}^{t-} h^m_{m'}(t-s)dN^{m'}_s\Big),
	\end{equation*} 
	provided that the link functions $\Phi^m$ are Lipschitz. Indeed, in this case the nonlinear Hawkes process can be stochastically dominated by the linear Hawkes process
	\begin{equation*}
		\lambda^m_t = \Phi^m(\mu_m) + \sum_{m'\in\M}\itg_{-\infty}^{t-} \Lip(\Phi^m) \vert h^m_{m'}\vert (t-s)dN^{m'}_s.
	\end{equation*}
\end{remark}

The following corollary is an immediate consequence of Theorem \ref{theorem2} and the previous remarks.

\begin{corollary}\label{cor3.14}
	Let $\bm{N}$ be a Hawkes process with parameters $\M,\bm{\mu},\h$. Let $\L$ denote the function introduced in Theorem \ref{theorem1} with $\H = (\Vone{h_{m,m'}})_{m,m'\in\M}$. Suppose that $\H \in \Ge{r}{K}$ and $\H^{\intercal} \in \Ge{r}{\tilde{K}}$ with $r<1$ and $K,\tilde{K} <\infty$. Then, for any $m\in\M$ and any $L>0$, the following holds,
	\begin{equation}
		\E{e^{tN^m([0,L))}} \leq \exp\Big( t L \vert \bm{\mu} \vert_{\infty} \big(1+\frac{2\tilde{K}}{1-r}\big) e^{(1+\frac{2K}{1-r})t}\Big) \leq \exp\Big( t L \vert \bm{\mu} \vert_{\infty} \big(1+\frac{2\tilde{K}}{1-r}\big) \frac{1+r}{2r} \Big),
	\end{equation}
	where $0\leq t\leq t_0(r,K) = \frac{\log((1+r)/2r)}{1+2K/(1-r)}$.
\end{corollary}

\begin{remark}\label{remark3.14}
	In a high-dimensional setting, if $\vert\bm{\mu}\vert_\infty$, $r$, $K$, and $\tilde K$ are independent of $M$, then the previous estimate provides dimension-free exponential moment bounds for the number of points of a fixed type.
\end{remark}

\subsection{Proof of Theorem \ref{theorem2}}\label{section3.4}

The cluster representation of Hawkes processes, introduced in \cite{hawkes_cluster_1974}, states the following. Let $(\Pi^m)_{m\in\M}$ be independent Poisson point processes on $\R$ with intensities $(\mu_m)_{m\in\M}$. Conditionally on $(\Pi^m)_{m\in\M}$, for every $m\in\M$ and every $x\in\Pi^m$, let $G_x^m$ be independent $\Pois(\h)$-clusters with root of type $m$ born at time $x$. Then the following holds,
\begin{equation}\label{cluster rep}
	(N^m)_{m\in\M} \overset{d}{=} \big(\{ s\in\R \mid \exists m'\in\M, \exists x\in\Pi^{m'}, \exists u\in \U \ \text{such that} \ (u,m,s)\in G^{m'}_x\}\big)_{m\in\M}.
\end{equation}

Recall that the clusters of Hawkes processes are homogeneous, i.e. $h_{m,m'}(s,t) = h_{m,m'}(t-s)$. Thus each cluster can be viewed as a $\Pois(\H)$-BGW tree together with a temporal embedding on $\R$ given by the birth times of its individuals.

For a cluster $G = (u,\tp(u),\bd(u))_{u\in\A}$, define
\[G(\bm{f}) = \sum_{u\in\A}f^{\tp(u)}(\bd(u)).\] 
The argument now follows the approach developed in \cite{leblanc_exponential_2024}. By the cluster representation and Campbell's formula for Poisson point processes,
\begin{align}
	\E{e^{\bm{N}(\bm{f})}} & = \E{\prod_{m\in\M} \prod_{x\in\Pi^m} \E{e^{G^m_x(\bm{f})}}}\nonumber \\
	& = \exp\left( \sum_{m\in \M} \itg_{\R} \big[\psi(x,m,\bm{f})-1\big] \mu_m dx\right) \label{eq6.8}
\end{align}
where
\begin{equation*}
	\psi(x,m,\bm{f}) = \E{ e^{G^m_x(\bm{f})}} = \E{e^{G^m_0(\bm{f}\circ \tau_x)}},
\end{equation*}
with $G^m_0$ be a $\Pois(\h)$-cluster with root of type $m$ and born at time $0$, and $\tau_x(s) = s+x$. Let $T>0$, decomposing the integral over $\R$ into translates of $[0,T)$ yields,
\begin{equation}\label{eq6.9}
	\itg_{\R} \big[\psi(x,m,\bm{f})-1\big] dx  = \itg_0^T \sum_{n\in\Z} \big[\psi(x+nT,m,\bm{f})-1\big] dx.
\end{equation}
To bound $\sum_{n\in\Z} \big[\psi(x+nT,m,\bm{f})-1\big]$ we use the same arguments as in \cite{leblanc_exponential_2024}, which we explain again for clarity.
\begin{align*}
	\sum_{n\in\Z} \big[ \psi(x+nT,m,\bm{f})-1\big] & = \E{ \sum_{n\in\Z} \big[e^{G^m_0(\bm{f}\circ \tau_{x+nT})}-1\big]} \\
	& \leq \E{ -1+\prod_{n\in\Z} e^{G^m_0(\bm{f}\circ \tau_{x+nT})}  } \\
	& = \E{e^{ G^m_0\big(\sum_{n\in\Z} \bm{f}\circ \tau_{x+nT}\big)}}-1.
\end{align*}
The second line comes from $\sum_n (x_n-1) \leq -1+\prod_n x_n$ for any $x_n\geq 1$. Let $m\in\M$, then by definition of $u_T(\bm{f})$, we have,
	\[\forall s\in\R, \ \Big(\sum_{n\in\Z} \bm{f}^m\circ \tau_{x+nT}\Big)(s) = \sum_{n\in\Z} f^m(s+x+nT) \leq u_T(\bm{f})_m.\]
	Thus it is clear that
\begin{equation*}
	G^m_0\big(\sum_{n\in\Z} \bm{f}\circ \tau_{x+nT}\big) \leq  u_T(\bm{f})\trp \card_{\M}(G^m_0) \overset{d}{=} u_T(\bm{f})\trp \card_{\M}(\T^m),
\end{equation*}
with $\T^m$ be a $\Pois(\H)$-BGW tree with root of type $m$. Therefore,
\begin{equation}\label{eq6.10}
	\sum_{n\in\Z} \big[\psi(x+nT,m,\bm{f})-1\big] \leq \exp\big( \e_{m}\trp \L(u_T(\bm{f}))\big)-1.
\end{equation}

Combining equations \eqref{eq6.8}, \eqref{eq6.9}, and \eqref{eq6.10} yields
\begin{equation*}
	\E{e^{\bm{N}(\bm{f})}} \leq \exp\Big[ T \bm{\mu}\trp \big(e^{\L(u_T(\bm{f}))}-\ones\big)\Big],
\end{equation*}
which is the result. Assume now that $\H \in\Ge{r}{K}$ with $r<1$, $K<\infty$ and suppose that $\vert u_T(\bm{f}) \vert_{\infty} \leq t_0(r,K)$. By \eqref{th1eq7} and \eqref{th1eq8} of Theorem \ref{theorem1} we have $\L(u_T(\bm{f})) \preceq \big( \Id-\frac{1+r}{2r}\H\big)^{-1}u_T(\bm{f})$, and $\vert\L(u_T(\bm{f}))\vert_{\infty} \leq \big(1+\frac{2K}{1-r}\big) \vert u_T(\bm{f})\vert_{\infty}$. Thus, by inequality $e^x-\ones \preceq e^{\vert x\vert_{\infty}} x, \ x\in\pos$, we obtain 
\begin{equation*}
	\bm{\mu}\trp \big(e^{\L(u_T(\bm{f}))}-\ones\big) \leq e^{(1+\frac{2K}{1-r}) \vert u_T(\bm{f})\vert_{\infty}} \bm{\mu}\trp \big(\Id-\frac{1+r}{2r}\H\big)^{-1} u_T(\bm{f}),
\end{equation*}
which is \eqref{quant2}. To obtain \eqref{quant3} from \eqref{quant2} we use the classical inequality $x\trp A y \leq \vert x\vert_\infty \Vvert A\Vvert_1 \vert y\vert_1$. Since $\Vvert A\Vvert_1 = \Vvert A\trp\Vvert_{\infty}$ for any matrix $A$, from assumption $\H\trp \in\Ge{r}{\tilde{K}}$ we have
\begin{align*}
	\Vvert \big(\Id-\frac{1+r}{2r}\H\big)^{-1} \Vvert_1  \leq 1 + \sum_{n\geq 1} \Vvert \big(\frac{1+r}{2r} \H\big)^n \Vvert_1 & = 1 + \sum_{n\geq 1} \big(\frac{1+r}{2r}\big)^n \Vvert  (\H^{\intercal})^n \Vvert_{\infty} \\
	& \leq 1 + \sum_{n\geq 1} \tilde{K} \big(\frac{1+r}{2}\big)^n\\
	& \leq  1 + \frac{2\tilde{K}}{1-r}.
\end{align*}
Which concludes the proof of \eqref{quant3} and Theorem \ref{theorem2}.

\section{Tail estimates for inhomogeneous Poisson clusters}\label{section4}

In this section we derive exponential tail estimates for inhomogeneous multitype Poisson clusters. The inhomogeneous setting naturally leads to a bivariate version of the classical convolution operator. 

\subsection{Bivariate convolutions and associated norms}\label{sec4.1}

Let $\bm{z}:\R \to \mathcal{M}_{n,p}(\R)$ be a measurable matrix-valued function. Define its matrix-valued $L^1$ and $L^\infty$ norms by
\[\Vone{\bm{z}} = \big(\Vone{z_{ij}}\big)_{i,j} \ , \quad \Vity{\bm{z}} = \big(\Vity{z_{ij}}\big)_{i,j}.\]

Consider now two bivariate measurable matrix-valued functions $\bm{v}:\R^2 \to \mathcal{M}_{n,p}(\R)$ and $\bm{w} : \R^2 \to \mathcal{M}_{p,q}(\R)$. Their bivariate convolution $\bm{v} \star \bm{w} : \R^2 \to \mathcal{M}_{n,q}(\R)$, when defined, is
\[\bm{v} \star \bm{w}(s,t) = \itg_{-\infty}^{\infty} \bm{v}(s,x) \bm{w}(x,t)dx, \quad s,t\in\R.\]
\begin{remark}
	If one takes $\bm{v}(s,t) = \bm{f}(t-s)$ and $\bm{w}(s,t) = \bm{g}(t-s)$ then one has $\bm{v} \star \bm{w}(s,t) = \bm{f} \star \bm{g}(t-s)$ the usual convolution for univariate matrix-valued functions.
\end{remark}

For a measurable bivariate matrix-valued function $\bm{v}$, define the following bivariate norms,
\[\Vityty{\bm{v}} = \Vity{ s\in\R\mapsto \Vity{\bm{v}(s,\cdot)}} = \big( \sup_{s,t \in\R} \vert v_{ij}(s,t)\vert\big)_{i,j}\]
and
\[ \Vityone{\bm{v}} = \Vity{s\in\R\mapsto \Vone{\bm{v}(s,\cdot)}} = \big(\sup_{s\in\R} \Vone{v_{ij}(s,\cdot)}\big)_{i,j}.\]
The residual integral of $\bm{v}$, denoted $\Res_{\bm{v}}:\R^2 \to \mathcal{M}_{n,p}(\R)$, quantifies for $s,t\in\R$, the mass of $\bm{v}(s,\cdot)$ remaining after time $t$. It is defined by
\[\Res_{\bm{v}}(s,t) = \itg_t^{\infty} \bm{v}(s,x)dx.\]
The uniform residual integral of $\bm{v}$ at delay $d \in\R$, denoted $\Res^\infty_{\bm{v}}(d)$, measures the maximal tail mass at delay $d$. It is defined by 
\[\Res^{\infty}_{\bm{v}}(d) = \Vert s\in\R \mapsto \Res_{\bm{v}}(s,s+d)\Vert_{L^{\infty}}.\]

\begin{remark}
	One has that $\Res_{\bm{v}} = \bm{v} \star (\ind_{D} \Id)$ where $D = \{(s,t) \mid t\leq s\}$.
\end{remark}

\subsection{Tail estimates for inhomogeneous Poisson clusters}

In Section \ref{section3} we derived exponential estimates for $\Pois(\H)$-BGW trees, and applied these results to Hawkes processes. In this section we study inhomogeneous Poisson clusters generated by measurable nonnegative bivariate functions $\h = (h_{ij})_{i,j\in\M}$. The matrix $\H$ is defined by $\H = \Vityone{\h}$ and assumed to have finite entries. Clusters consist in $\M$-type random trees with birth dates attached to each individuals. 

An inhomogeneous Poisson cluster can be viewed as a multitype branching tree together with birth times attached to each individual. If one ignores the temporal component, the genealogical structure is stochastically dominated by a $\Pois(\H)$-BGW tree. Consequently, the exponential estimates derived in Section \ref{section3} immediately transfer to the total progeny of clusters.

The main additional difficulty in the inhomogeneous setting is the temporal distribution of the birth times. The purpose of this section is therefore to quantify the number of descendants born far from the root time, namely in intervals of the form $[t,\infty)$ with $t$ possibly large. Such estimates provide quantitative control of long-range temporal dependence in Hawkes processes.

Theorem \ref{theorem4.5} below provides quantitative estimates on the temporal tail of inhomogeneous Poisson clusters. More precisely, for a cluster born at time $s\in\R$, we study the number of descendants whose birth times occur after time $t\in\R$. The key idea is that the genealogy of clusters is controlled by a multitype $\Pois(\H)$-BGW tree, while the temporal locations of the descendants are governed by the functions $\h$. The corresponding log-Laplace transforms satisfy a fixed-point equation \eqref{eq4.2} reflecting the branching property of clusters. Contributions to the tail come either from the root itself (when $t\leq s$), or from descendants generated recursively through the interactions $\h$. The bounds obtained in Theorem \ref{theorem4.5} show that the decay of the cluster tail is directly inherited from the decay properties of $\h$. In particular:
\begin{itemize}
    \item exponentially decaying interactions lead to exponentially decaying cluster tails,
    \item polynomially decaying interactions lead to polynomial cluster tails,
    \item compactly supported interactions yield geometric-type decay controlled by powers of $\H$.
\end{itemize}
The proof of Theorem \ref{theorem4.5} is given in Section \ref{section4.3}.

\begin{theorem}\label{theorem4.5}
    Let $G_s^m$ be a $\Pois(\h)$-cluster with root of type $m\in\M$ born at time $s\in\R$. For $u\in \pos$ and $t\in\R$ define the tail log-Laplace transform by,
\[\Ta_{u,m}(s,t) = \log\bigg( \E{e^{u\trp \card_{\M}\big(G^{m}_{s} \cap [t,\infty)\big)}} \bigg).\]
Then, $\bTa{u}=(\Ta_{u,m})_{m\in\M}$ satisfies the following fixed-point equation,
\begin{equation}\label{eq4.2}
        \bTa{u}(s,t) =  u \ind_{t\leq s}  + \Big[\h \star \big(e^{\bTa{u}}-\ones \big)\Big](s,t), \quad s,t\in\R.
\end{equation}
	Consider $\L(u)$ from Theorem \ref{theorem1} with $\H = \Vityone{\h}$. Then we have $\Vityty{\bTa{u}} \preceq \L(u)$ and the following bound holds,
	\begin{equation}\label{eq4.3}
		\bTa{u}(s,t) \preceq  u \ind_{t\leq s} + \Res_{\bm{\Psi}_u}(s,t) u, \quad s,t\in\R,
	\end{equation}
where $\displaystyle \bm{\Psi}_u = \sum_{n=1}^{\infty}\big( \h \diag(e^{\L(u)}) \big)^{\star n}$, and $\cdot^{\star n}$ denotes the $n$-fold bivariate convolution. In particular the following holds.
\begin{enumerate}
	\item Suppose that $\H \diag(e^{\L(u)}) \in \Ge{\alpha}{K}$ for some $\alpha <1$ and that \[\Res_{\h}^{\infty}(d) \preceq A e^{-cd} \H, \quad \forall d\geq 0,\] for some $A,c>0$. Then for all $0<\varepsilon< \frac{c(1-\alpha)}{(1-\alpha)+A\alpha}$ we have \[\vert\bTa{u}(s,s+d)\vert_{\infty} \leq   B_{\varepsilon} e^{-\varepsilon d} \vert u \vert_{\infty}, \quad \forall s\in\R, \ \forall d>0,\]
	with $B_{\varepsilon}= \frac{K\alpha\big(1+\frac{A\varepsilon}{c-\varepsilon}\big)}{1-\alpha\big(1+\frac{A\varepsilon}{c-\varepsilon}\big)}$. 
	\item Suppose that $\H \diag(e^{\L(u)}) \in \Ge{\alpha}{K}$ for some $\alpha <1$ and that \[\Res_{\h}^{\infty}(d) \preceq \frac{A}{(1+d)^{\gamma}}\H, \quad \forall d \geq 0,\] for some $A,\gamma>0$. Then for any $0<\delta<1$ we have \[\vert\bTa{u}(s,s+d)\vert_{\infty} \leq \frac{B_{\delta} \vert u\vert_{\infty}}{(1+d)^{\gamma(1-\delta)}}, \quad \forall s\in\R, \ \forall d> 0,\]
	with $B_{\delta}=\frac{\alpha K(K+1) }{(1-\alpha)^2}  \Big[A\big(1+\frac{\gamma}{\delta\log(1/\alpha)}\big)^{\gamma}+1\Big]$. 
	\item Suppose that $\spr(\H)<1$, that $\vert \L(u)\vert_{\infty}<\infty$, and that $\Res_{\h}^{\infty}(A) = 0$ for some $A>0$, equivalently, $\h$ vanishes almost everywhere outside $\{(s,t)\in\R^2\mid t\leq s+A\}$. Then there exists a constant $C$ depending on $u,\H$ and $A$ such that \[\bTa{u}(s,s+d) \preceq C \H^{\lfloor d/A \rfloor} \L(u), \quad \forall d\geq 0.\]
\end{enumerate}
If $\H\in\Ge{r}{K}$ with $r<1$ and if $\vert u\vert_{\infty} \leq t_0(r,K)$ from Theorem \ref{theorem1}, then in points $(1)$ and $(2)$ we have $\H \diag(e^{\L(u)}) \in \Ge{\alpha}{K}$ with $\alpha = \frac{1+r}{2}$, and in point $(3)$ one can take $C = \big(\frac{1+r}{2r}\big)^{1+2K/(1-r)}$.
\end{theorem}

\begin{remark}
	Theorem \ref{theorem4.5} shows that the temporal decay of cluster tails inherits the decay properties of the interaction functions $\h$. Exponential decay of $\h$ yields exponential tail bounds, polynomial decay yields polynomial tail bounds, and compact support of $\h$ leads to geometric decay governed by powers of $\H$.
\end{remark}

\begin{remark}
	The matrix $\H\diag(e^{\L(u)})$ plays the role of an exponentially tilted offspring matrix. It naturally appears when controlling exponential moments recursively along the branching structure.
\end{remark}

\begin{remark}
	If $\Res_{\h}^{\infty}(A) = 0$ for some $A>0$, then from Equation \eqref{eq4.3} and direct support considerations, one can derive that 
	\begin{equation*}
		\bTa{u}(s,s+d) \preceq \sum_{n\geq \lfloor d/A \rfloor}\big(\H\diag(e^{\L(u)})\big)^{n} u \preceq \big(\H\diag(e^{\L(u)})\big)^{\lfloor d/A \rfloor} B u, \quad \forall d\geq 0,
	\end{equation*}
	with $B = \big(\Id-\H\diag(e^{\L(u)})\big)^{-1}$. Thus point $(3)$ is asymptotically sharper since we have $\spr(\H) \leq \spr\big(\H\diag(e^{\L(u)})\big)$.
\end{remark}

\begin{remark}
	The condition $\H \diag(e^{\L(u)}) \in \Ge{\alpha}{K}$ with $\alpha <1$ is equivalent to the condition $\spr\big( \H \diag(e^{\L(u)}) \big) <1$ and thus, as mentioned in Remark \ref{remark3.3}, Theorem \ref{thm exact multD} proves that if $\spr(\H)<1$ and $\L(u)$ is finite then $\spr\big( \H \diag(e^{\L(u)}) \big) <1$ except if $u$ is a boundary point.
\end{remark}

\subsection{Proof of Theorem \ref{theorem4.5}}\label{section4.3}

The proof follows two main steps. First, we derive a recursive fixed-point equation for the tail log-Laplace transform using the branching property of Poisson clusters. Second, we analyze this equation through an iterative scheme and bivariate convolution estimates, which allows us to transfer the decay properties of $\h$ to tails of clusters. This second step relies on several technical results on bivariate convolutions collected in Appendix \ref{appendixA}.

Let $u\in\pos$ and $G^{m}_{s}$ be an inhomogeneous $\Pois(\h)$-cluster with root of type $m$ and born at time $s$. Recall that for $m\in\M, \ s\in\R$ and $t\in\R$ the function $\Ta_{u,m}$ is defined by
    \[\Ta_{u,m}(s,t) = \log\Big( \E{e^{u\trp \card_{\M}\big(G^{m}_{s} \cap [t,\infty)\big)}} \Big).\]
    Recall that, conditionally on the first generation, each individual in that generation independently generates a $\Pois(\h)$-cluster of which it is the root. This branching decomposition yields a recursive equation for the tail log-Laplace transform.
    Fix $s\in\R$ and $m\in\M$. For $m'\in\M$, let $X_{m,m'}(s)\sim \Pois(\Vert h_{m,m'}(s,\cdot)\Vert_1)$, and $Z^{m',k}\sim \frac{h_{m,m'}(s,w)}{\Vert h_{m,m'}(s,\cdot)\Vert_1} dw$, $k\in\N^*$, where all variables are independent. Conditionally on these variables, let $G^{m'}_{Z^{m',k}}$ be independent $\Pois(\h)$-clusters with root of type $m'$ and born at time $Z^{m',k}$. We have    \begin{equation*}
        e^{\Ta_{u,m}(s,t)} = \E{e^{u\trp \card_{\M}\big(G^{m}_{s} \cap [t,\infty)\big)}}  = \E{e^{u\trp \e_m \ind_{t\leq s}} \prod_{m'\in\M} \prod_{k = 1}^{X_{m,m'}(s)} e^{u\trp \card_{\M}\big(G^{m'}_{Z^{m',k}}\cap[t,\infty)\big)}}.
    \end{equation*}
    By independence we first integrate over the clusters and then over the birth dates $Z^{m',k}$, yielding,
    \begin{equation*}
    	e^{\Ta_{u,m}(s,t)}  = \E{e^{u\trp \e_m \ind_{t\leq s}}\prod_{m'\in\M} \prod_{k = 1}^{X_{m,m'}(s)} \E{e^{\Ta_{u,m'}(Z^{m'},t)}}},
    \end{equation*}
    where $Z^{m'}\sim Z^{m',1}$. Note that by definition of $Z^{m'}$ we have
    \begin{equation*}
    	\E{e^{\Ta_{u,m'}(Z^{m'},t)}} = \itg_{\R} \frac{h_{m,m'}(s,w)}{\Vert h_{m,m'}(s,\cdot)\Vert_1} e^{\Ta_{u,m'}(w,t)} dw = \frac{\itg_{\R} h_{m,m'}(s,w) \big(e^{\Ta_{u,m'}(w,t)} -1\big) dw}{\Vert h_{m,m'}(s,\cdot)\Vert_1}+1.
    \end{equation*}
    Thus integrating over the $X$'s together with the mgf of the Poisson distribution give,
    \begin{align*}
        e^{\Ta_{u,m}(s,t)} & = \exp\Big[ u\trp \e_m \ind_{t\leq s} + \sum_{m'} \itg_{\R} h_{m,m'}(s,w) \big(e^{\Ta_{u,m'}(w,t)}-1\big) dw\Big]\\
        & = \exp\Big[ \Big(u \ind_{t\leq s} +\big[\h\star \big(e^{\bTa{u}}-\ones \big)\big](s,t)\Big)_{m}\Big],
    \end{align*}
    which is the fixed-point equation \eqref{eq4.2}. To derive quantitative bounds from the fixed-point equation, we introduce the Picard iterates associated with \eqref{eq4.2}. Define $\bTa{u}^0$ by $\bTa{u}^0(s,t) = 0$, $s,t\in\R$, and for $n\geq 1$ let $\bTa{u}^n$ defined by
    \begin{equation}\label{demth4eq1}
    	\bTa{u}^n(s,t) = u \ind_{t\leq s} + \big[\h\star \big(e^{\bTa{u}^{n-1}}-\ones \big)\big](s,t), \quad s,t\in\R.
    \end{equation}
    As in the proof of Theorem \ref{theorem1}, the quantity $\bTa{u}^{n+1}$ corresponds to the log-Laplace transform obtained by truncating the cluster after generation $n$. The monotone convergence theorem yields $\bTa{u}^n(s,t) \xrightarrow[n\to\infty]{} \bTa{u}(s,t)$ for any $s,t\in\R$.
    Thus we may write
    \begin{equation*}
    	\bTa{u}  = \sum_{n\geq 0} (\bTa{u}^{n+1}-\bTa{u}^n).
    \end{equation*}
    Since the underlying $\M$-type trees of clusters are stochastically dominated by $\Pois(\H)$-BGW trees with $\H = \Vityone{\h}$ it is clear that $\bTa{u}^n(s,t) \preceq \L(u)$ for any $n\geq 0$ and  any $s,t\in\R$. Thus, for any $n\geq 1$, \eqref{demth4eq1} combined with inequality $e^x-e^y\preceq \diag(e^{x})(x-y)$ for any $y\preceq x$ yields,  
    \begin{align*}
    	\bTa{u}^{n+1}-\bTa{u}^n & = \h\star(e^{\bTa{u}^n}-e^{\bTa{u}^{n-1}})\\
    	& \preceq \h\diag(e^{\L(u)})\star (\bTa{u}^n-\bTa{u}^{n-1})\\
    	& \preceq \Big(\h\diag(e^{\L(u)})\Big)^{\star n} \star (\bTa{u}^{1}-\bTa{u}^{0})\\
    	& = \Big(\h\diag(e^{\L(u)})\Big)^{\star n} \star u\ind_{D},
    \end{align*}
    where $\ind_{D}(s,t) = \ind{t\leq s}$. It follows that,
    \begin{align*}
    	\bTa{u} & = \sum_{n\geq 0} (\bTa{u}^{n+1}-\bTa{u}^n)\\
    	& \preceq u\ind_{D} + \sum_{n\geq 1} \big(\h\diag(e^{\L(u)})\big)^{\star n} \star u\ind_{D}\\
    	& = u\ind_{D} + \bm{\Psi}_u \star \ind_{D} \Id  u\\
    	& = u\ind_{D} + \Res_{\bm{\Psi}_u} u,\\
    \end{align*}
    which proves \eqref{eq4.3}.
    
    Finally let us prove points $(1)$, $(2)$, and $(3)$. Points $(1)$ and $(2)$ are straightforward from Proposition \ref{prop4.4} and Proposition \ref{prop4.3} respectively. And, in the case where $\H \in \Ge{r}{K}$ and $\vert u\vert_{\infty} \leq t_0(r,K)$, we have that $\vert \L(u)\vert_{\infty} \leq \log((1+r)/2r)$ and thus it is clear that $\H\diag(e^{\L(u)}) \in\Ge{\alpha}{K}$ with $\alpha = \frac{1+r}{2}$. It concludes the proof of points $(1)$ and $(2)$ by the above mentioned Propositions.
    
    To prove point $(3)$, we discretize time at scale $A$ and study the decay of the sequence
\[\g_u(n)=\Vity{s\in\R \mapsto \bTa{u}(s,s+nA)}, \quad n\in\N\]
which measures the tail at distance $nA$ from the root. Since the interval $[t,\infty)$ decreases as $t$ increases, the map $t\mapsto \bTa{u}(s,t)$ is non-increasing. Thus, since $\h$ is supported in $\{(s,t)\in\R^2 \mid t\leq s+A\}$, we have for $n\in\N^*$,
\begin{align*}
	\bTa{u}(s,s+nA) & =  \itg_{-\infty}^{s+A} \h(s,w) (e^{\bTa{u}(w,s+nA)}-\ones)dw\\
	& \preceq  \itg_{-\infty}^{s+A} \h(s,w) (e^{\bTa{u}(w,w+(n-1)A)}-\ones)dw\\
	& \preceq  \itg_{-\infty}^{s+A} \h(s,w) (e^{\g_u(n-1)}-\ones) dw \\
	& \preceq  \H (e^{\g_u(n-1)}-\ones).
\end{align*}
Thus we have
\[\forall n\in\N, \quad \g_u(n+1) \preceq \H(e^{\g_u(n)}-\ones), \quad   \text{and} \quad \g_u(0)\preceq \L(u).\]
Using inequality $e^x-\ones \preceq e^{\vert x\vert_{\infty}} x$ for $x\succeq 0$ together with an induction, one has, for all $n\geq 0$,
\begin{equation}\label{demth4eq2}
    \g_{u}(n) \preceq e^{\sum_{k=0}^{n-1}\vert \g_{u}(k)\vert_{\infty}} \H^n \L(u).
\end{equation}
Let $m\in\M$, by definition of $\H = \Vityone{\h}$ and by the support assumption, it follows that for any $s\in\R$ and any $n\in\N$, the random variables $e^{u\trp \card_{\M}(G_s^m\cap[s+nA,\infty))}$ are all stochastically dominated by $e^{u\trp \card_{\M}(\T^m_{\geq n})}$ where $\T^m$ is a $\Pois(\H)$-BGW tree with root of type $m$. Since $\L(u)$ is finite, $e^{u\trp \card_{\M}(\T^m)}$ is integrable, thus by dominated convergence theorem we have
\begin{equation*}
	\g_u(n) \preceq \Big(\E{e^{u\trp \card_{\M}(\T^m_{\geq n})}}\Big)_{m\in\M} \xrightarrow[n\to\infty]{} 0.
\end{equation*}
Cesàro-type averaging implies that for any $\varepsilon >0$ there exists $C=C_{\varepsilon}$ (depending on $u,\H$ and $A$) such that
\begin{equation*}
	\sum_{k=0}^{n-1}\vert \g_{u}(k)\vert_{\infty} \leq \varepsilon n + C_{\varepsilon}, \quad n\geq 0.
\end{equation*}

Choose $\varepsilon>0$ small enough so that $\spr(\H)< e^{-\varepsilon}$. Since $\Vvert \H^n \Vvert_{\infty}^{1/n} \xrightarrow[]{n\to\infty} \spr(\H)$, by \eqref{demth4eq2} we have,
\begin{align*}
	\vert \g_{u}(n)\vert_{\infty} & \leq e^{C_{\varepsilon}} e^{\varepsilon n} \Vvert \H^n\Vvert_{\infty} \vert \L(u)\vert_{\infty} \\
	& \leq C \alpha^n
\end{align*}
for some $C$ and $\alpha<1$ depending on $u$, $\H$, and $A$. With this new bound we have
\begin{equation*}
	\sum_{k=0}^{\infty}\vert \g_{u}(k)\vert_{\infty} <\infty,
\end{equation*}
which proves point $(3)$ by plugging in \eqref{demth4eq2}.

Finally, suppose that $\H\in\Ge{r}{K}$ with $r<1$, and let us prove that $C = \Big(\frac{1+r}{2r}\Big)^{1+2K/(1-r)}$ satisfies $\g_u(n) \preceq C \H^n \L(u)$ in point $(3)$.

One has $\vert \g_{u}(n)\vert_{\infty} \leq \vert \L(u) \vert_{\infty} \leq \log\big(\frac{1+r}{2r}\big)$ for any $n\geq 0$. Plugging in \eqref{demth4eq2} leads to,
\begin{equation*}
    \g_{u}(n)  \preceq  \Big(\frac{1+r}{2r}\H\Big)^n \L(u).
\end{equation*}
Thus since $\H\in\Ge{r}{K}$ we have for $n\geq 1$,
\begin{equation*}
    \vert \g_{u}(n)\vert_{\infty} \leq K \Big(\frac{1+r}{2}\Big)^n \log\big(\frac{1+r}{2r}\big).
\end{equation*}
Again, plugging this new inequality in \eqref{demth4eq2} leads to
\begin{align*}
    \g_{u}(n) &  \preceq \Big(\frac{1+r}{2r}\Big)^{1+K\sum_{k=1}^{n-1} \big(\frac{1+r}{2}\big)^k}\H^n \L(u) \\
    & \preceq \Big(\frac{1+r}{2r}\Big)^{1+2K/(1-r)}\H^n \L(u),
\end{align*}

which is the result. Theorem \ref{theorem4.5} is proved.

\subsection{Clusters with general kernel distributions}

The framework developed above naturally extends beyond kernels admitting densities with respect to the Lebesgue measure. In many situations, offspring birth times may contain deterministic delays, atomic components, or more singular distributions. Introducing kernel-valued interactions allows us to treat these situations in a unified way while preserving the branching structure and convolution framework developed previously.

In the previous framework, for fixed $s\in\R$, the function $h_{ij}(s,\cdot)$ represents, up to normalization, the density of the birth time of a child of type $j$ produced by a parent of type $i$ born at time $s$. The natural generalization therefore consists in replacing these densities by arbitrary nonnegative kernels. To avoid any confusion we denote $\varphi_{ij}$ these objects and $\bm{\varphi} = (\varphi_{ij})_{i,j\in\M}$. In the sequel, for any $i,j\in\M$, $\varphi_{ij}$ is a kernel, meaning that $\varphi_{ij} : \R\times \mathcal{B}(\R) \longrightarrow \R_+$, with $\mathcal{B}(\R)$ the Borel sigma-algebra, is such that for any $s\in\R$, $A \in\mathcal{B}(\R) \mapsto \varphi_{ij}(s,A)$ is a measure, and for any $A\in \mathcal{B}(\R)$, $s \in\R \mapsto \varphi_{ij}(s,A)$ is measurable.

The matrix $\H$ is similarly defined by $H_{ij} = \sup_{s\in\R}\varphi_{ij}(s,\R)$ and assumed to have finite entries. A $\Pois(\bm{\varphi})$-cluster is then defined by the following:
\begin{itemize}
	\item Individuals reproduces independently.
	\item Given an individual $a$ of type $i$ and born at time $s$, independently for any $j$, $a$ has $\Pois(\varphi_{ij}(s,\R))$ children of type $j$.
	\item Independently, for any child $b$ of $a$ of type $j$, we have \[\bd(b) \sim \frac{\varphi_{ij}(s,\cdot)}{\varphi_{ij}(s,\R)}.\]
\end{itemize}

Again, if one ignores the birth dates, the underlying genealogical tree is stochastically dominated by a $\Pois(\H)$-BGW tree

The convolution product naturally extends to kernels. If $\phi$ and $\psi$ are kernels, their convolution $\phi\star\psi$ is also a kernel and defined by
\begin{equation*}
	\phi \star \psi(s,A) = \itg \phi(s,dx)\psi(x,A).
\end{equation*}
Similarly, the convolution between a kernel $\phi$ and a measurable bivariate function $f$, is a measurable bivariate function defined as follows, 
\begin{equation*}
	\phi \star f (s,t) = \itg_{\R} \phi(s,dx) f(x,t), \quad \forall s,t\in\R.
\end{equation*}

With these definitions, Theorem \ref{theorem4.5} can be extended in the exact same formulation. The extension is Theorem \ref{theorem4.6} stated below.

\begin{theorem}\label{theorem4.6}
    Let $G^m_s$ be a $\Pois(\bm{\varphi})$-cluster with root of type $m\in\M$ born at time $s\in\R$. Let $u\in \pos$ and for $t\in\R$ define
\[\Ta_{u,m}(s,t) = \log\bigg( \E{e^{u\trp \card_{\M}\big(G^{m}_s \cap [t,\infty)\big)}} \bigg).\]
Then, $\bTa{u}=(\Ta_{u,m})_{m\in\M}$ satisfies the following fixed-point equation,
\begin{equation}\label{eq4.6}
        \bTa{u}(s,t) =  u \ind_{t\leq s}  + \Big[\bm{\varphi} \star \big(e^{\bTa{u}}-\ones \big)\Big](s,t), \quad s,t\in\R.
\end{equation}
	Consider $\L(u)$ from Theorem \ref{theorem1} with $\H = \big(\sup_{s\in\R}\varphi_{ij}(s,\R)\big)_{i,j\in\M}$ and suppose that $\vert \L(u)\vert_{\infty}<\infty$. Then for any $s,t\in\R$ the following bound holds,
	\begin{equation}\label{eq4.5}
		\bTa{u}(s,t) \preceq  u\ind_{t\leq s} + \bm{\Psi}_u\big(s,[t,\infty)\big)u,
	\end{equation}
where the kernel $\bm{\Psi}_u$ is defined by $\displaystyle\bm{\Psi}_u = \sum_{n=1}^{\infty}\big( \bm{\varphi} \diag(e^{\L(u)}) \big)^{\star n}$.\\
Points $(1),(2)$ of Theorem \ref{theorem4.5} also hold up to replacing, for any $d\geq 0$, $\Res^{\infty}_{\h}(d)$ by $\big(\sup_{s\in\R}\varphi_{ij}(s,[s+d,\infty))\big)_{i,j\in\M}$, and for point $(3)$, one should ask $\varphi_{ij}(s,\cdot)$ to be supported in $(-\infty,s+A]$ for any $s\in\R$ and any $i,j\in\M$.
\end{theorem}

The proof follows exactly the same lines as the proof of Theorem \ref{theorem4.5}, replacing Lebesgue integrals by integrations with respect to kernels. We do not write it again.

\subsection{Tail estimates for BGW trees}\label{sec3.3}

An important application of the kernel framework is the study of tails of multitype BGW tree beyond generation $n$ for integers $n\in\N$. Indeed, deterministic kernels allow one to encode generations as birth dates, thereby transforming statements on cluster tails into statements on descendants living at large generations.

Let $\H$ a nonnegative $\M\times\M$ matrix. Consider the special homogeneous kernel $\bm{\varphi}=\delta_1\H$,
where $\delta_1$ denotes the Dirac kernel at delay $1$:
\[\delta_1(s,A) = \ind_{s+1\in A}, \quad \varphi_{ij}(s,A) = H_{ij} \ind_{s+1\in A}, \quad s\in\R, \ A\in\mathcal{B}(\R), \ i,j\in\M.\]

In this case, every child is born exactly one unit of time after its parent. Consequently, individuals of generation $n$ are born exactly at time $s+n$ when the root is born at time $s$.

\begin{theorem}\label{theorem3}
    Let $\T^{m}$ be a $\Pois(\H)$-BGW tree with root of type $m\in\M$. Let $u\in\pos $. For $n\in\N$, define the vector $\Re_n(u)\in [0,\infty]^{\M}$ by
    \begin{equation}\label{th3eq1}
        \forall n\in\N, \ \forall m\in\M, \quad \exp\big( \e_{m}\trp \Re_n(u)\big) = \E{e^{u\trp \card_{\M}(\T^{m}_{\geq n})}}.
    \end{equation}
    Then the sequence $(\Re_n(u))_{n\in\N}$ satisfies the following recursion,
    \begin{equation}\label{th3eq2}
        \Re_0(u) = \L(u) \quad \text{and} \quad \Re_{n+1}(u) = \H \big(e^{\Re_{n}(u)}-\ones\big), \quad n\geq 0,
    \end{equation}
    where $\L(u)$ is defined in Theorem \ref{theorem1}. If we suppose that $\spr(\H)<1$ and $\vert \L(u)\vert_{\infty} < \infty$, then there exists a constant $C$ depending on both $u$ and $\H$ such that for all $n\in\N$,
    \begin{equation}\label{th3eq4}
    	\Re_n(u) \preceq C \H^n \L(u).
    \end{equation}
    More precisely, if we suppose that $\H\in\Ge{r}{K}$ with $r<1$, then for all $u\in\pos$ with $\vert u\vert_{\infty} \leq t_0(r,K)$, defined in Theorem \ref{theorem1}, and for all $n\geq 0$, we have
    \begin{equation}\label{th3eq5}
        \Re_n(u) \preceq \Big(\frac{1+r}{2r}\Big)^{1+\frac{2K}{1-r}} \H^n \L(u).
    \end{equation}
\end{theorem}

\begin{remark}
	As in the previous sections, the sequence $(\Re_n(u))_{n\in\N}$ is entirely characterized by a nonlinear fixed-point recursion. This recursion encodes the exponential moments of descendants surviving beyond generation $n$. Note also that for any $n\geq 0$, $m_0\in\M$ and $u\in\pos$, we have
    \begin{equation*}
    	\e_{m_0}\trp \Re_n(u) = \infty \quad \iff \quad \exists m\in\M, \ (\H^n)_{m_0,m} >0 \ \text{and} \ \e_m\trp\L(u) = \infty.
    \end{equation*}
    Indeed, if the RHS holds, then there exists a probability $p>0$ of having at least one individual of type $m$ at generation $n$ from a single root of type $m_0$. Conditionally on this event, the total progeny of the individual of type $m$ at generation $n$ is distributed as $\card_{\M}(\T^{m})$ and thus $e^{\e_{m_0}\trp\Re_n(u)} \geq p e^{ \e_m\trp\L(u)} = \infty$. The contrapositive of the converse follows immediately from equation \eqref{th3eq4}.
\end{remark}

\begin{remark}
	By definition the sequence $(\Re_n(u))_{n\in\N}$ is decreasing, thus for any $n\geq 0$ the following holds,
    \begin{equation}\label{th3eq3}
		\Re_{n+1}(u) \preceq \Re_n(u) \preceq \Re_0(u) = \L(u).
	\end{equation}
\end{remark}

\begin{remark}
	If $\H\in\Ge{r}{K}$ with $r<1$ and if $u\in\pos$ with $\vert u\vert_{\infty} \leq t_0(r,K)$, since $\Vvert \H^n\Vvert_{\infty} \leq K r^n$ with $r<1$ we have that $\vert \Re_n(u) \vert_{\infty} \leq  K' r^n \vert u\vert_{\infty}$ with $K'$ a constant. Therefore $\Re_n(u)$ decays exponentially fast.
\end{remark}

\begin{proof}[Proof of Theorem \ref{theorem3}]
	We apply Theorem \ref{theorem4.6} to the deterministic and homogeneous kernel $\bm{\varphi}=\delta_1\H$ with $\delta_1(s,A) = \ind_{s+1\in A}$. As explained previously, the birth date of an individual of the $n$-th generation of a $\Pois(\bm{\varphi})$-cluster born at time $s$ is $s+n$. Thus a $\Pois(\bm{\varphi})$-cluster born at time $s$ is a $\Pois(\H)$-BGW tree with additional birth dates given by $\bd(a) = s+\vert a\vert$ for any individual $a$. It follows that 
	\begin{equation*}
		\forall s\in\R, \forall n\in\N, \quad \bTa{u}(s,s+n) = \Re_n(u).
	\end{equation*}
	For $n\in\N^*$, Theorem \ref{theorem4.6} states that
	\begin{equation*}
		\Re_n(u) = \bTa{u}(0,n) = \Big[\bm{\varphi} \star \big(e^{\bTa{u}}-\ones \big)\Big](0,n) = \H \big(e^{\bTa{u}(1,n)}-\ones \big) = \H \big(e^{\Re_{n-1}(u)}-\ones \big).
	\end{equation*}
	Finally the bounds \eqref{th3eq4} and \eqref{th3eq5} follows directly from point $(3)$ of Theorem \ref{theorem4.6} applied with $A=1$.
\end{proof}

\section{Exact properties of the Laplace transform of Poisson BGW trees.}\label{section5}

\subsection{Main result on the Laplace transform of Poisson BGW trees}

This section studies the structure of the exponential moments of multitype $\Pois(\H)$-BGW trees and provides a multitype extension of Theorem \ref{theorem3.6}. In the one-dimensional setting, there is only one possible direction for the parameter $u$, and the geometry of the domain of finiteness is therefore relatively simple. In the multitype setting, the situation becomes substantially richer due to the interaction between types. Nevertheless, we obtain a rather complete description of the domain of finiteness and of the associated fixed-point equation, summarized in Theorem \ref{thm exact multD}.
 
In \cite{karim_compound_2025}, the authors consider general weights $u\in\R^{\M}$ and derive an implicit characterization of the finiteness domain. Restricting our attention to nonnegative weights $u\in\pos$ allows us to obtain a more explicit and more precise description of this domain.

We recall that $\H$ is a $\M\times \M$ matrix with finite nonnegative entries and for any $u\in\pos$ the vector $\L(u)$ is defined by
\[\exp\big(\e_m \trp \L(u)\big) = \E{e^{u\trp \card_{\M}(\T^m)}}, \]
where $\T^m$ is a $\Pois(\H)$-BGW tree with root of type $m\in\M$, see Theorem \ref{theorem1}. 

The following definition introduces key objects for the multitype study.

\begin{definition}
	Let
	\[E = \{u\in\pos \mid \vert \L(u)\vert_\infty <\infty\}.\]
	Equivalently, by Theorem \ref{theorem1}, $u\in E$ if and only if there exists $x\in\pos$ such that $x=u+\H(e^x-\ones)$. Define also 
	\begin{itemize}
		\item $\delta E := \big\{ u \in E \mid  \forall t>1, \ tu \notin E\big\}$,
		\item $\mathring{E} := E\setminus \delta E \subset \R_+^{\M}$,
		\item $\mathring{E}^* = \{ u\in \mathring{E} \mid u\succ 0\} \subset (0,\infty)^{\M}$.
	\end{itemize}
	Finally, for any set $F\subset \R^{\M}$ define the nonnegative lower cone of $F$ by $\CE(F) = \{ u\in\pos \mid \exists v\in F, \ u\preceq v\}$.
\end{definition}

\begin{remark}
	These notation are inspired by topology and, when $\spr(\H)<1$, we will show that they coincide with the usual notions of interior and boundary. Figure \ref{fig:fig1} illustrates the sets $\delta E$, $\mathring{E}$ and $\mathring{E}^*$. Since $\L$ is increasing with respect to $\preceq$, whenever $u\in E$ we also have $tu\in E$ for every $0\leq t\leq 1$. Thus $E$ is star-shaped with respect to the origin, motivating the definition $\delta E$ for its boundary part and $\mathring{E}=E\setminus\delta E$ for its interior part. We also introduce $\mathring{E}^*$ since, as observed in Lemma \ref{lemma3.5}, the existence of some $u\succ 0$ such that $\L(u)$ is finite carries significant information on the structure of $\H$.
\end{remark}

Irreducibility plays an important role, see Lemma \ref{lemma3.5}. When $\H$ is reducible, the structure of the interactions between types becomes important. The following definition introduces the relevant connectivity notions.

\begin{definition}\label{definition5.3}
	For $m\in\M$ define $\C(m) \subset \M$ the forward connected component of $m$ by
	\begin{equation}
		\C(m) = \{ m'\in\M \mid \exists N \geq 0, (\H^N)_{m,m'} >0\}.
	\end{equation}
	In particular $m\in\C(m)$. For $m'\neq m$, then $m'$ belongs to $\C(m)$ if and only if it is possible for an individual of type $m$ to have a descendant of type $m'$, possibly after more than one generation. For any $m\in\M$ we define
	\begin{equation}
		\H_{\Vert \C(m)} = \big( H_{ij} \big)_{i,j \in \C(m)} \in \R^{\C(m)^2}.
	\end{equation}
\end{definition}

\begin{remark}\label{rmk7.8}
	For any $m\in\M$ and any vector $x\in\R^{\M}$ with support in $\C(m)$, then $\H x$ has support in $\C(m)$ and if we denote $x_{\vert \C(m)} = (x_i)_{i\in\C(m)}$ we have the following identity in $\R^{\C(m)}$, \[ \H_{\Vert \C(m)} x_{\vert \C(m)} = \big(\H x\big)_{\vert\C(m)}.\]
	The following also holds: for any $m,m'\in\M$, if $m'\in\C(m)$ then $\C(m') \subset \C(m)$. Also, if $\H$ is irreducible then $\C(m)=\M$ for any $m\in\M$.
\end{remark}

The following Lemma is clear given Definition \ref{definition5.3}.

\begin{lemma}\label{lemma5.5}
	Let $m\in\M$, and $u\in\pos$ with support in $\C(m)$. Then $\L(u)$ also has support in $\C(m)$ and if we denote $\L_{(\H_{\Vert \C(m)},\C(m))}$ the equivalent of $\L$ but with $(\H_{\Vert \C(m)},\C(m))$ instead of $(\H,\M)$ we have the following equality,
	\[\big(\L(u)\big)_{\vert \C(m)} = \L_{(\H_{\Vert \C(m)},\C(m))}\big(u_{\vert \C(m)}\big).\]
\end{lemma}

Lemma \ref{lemma5.5} follows directly from the fact that descendants of a type in $\C(m)$ necessarily remain in $\C(m)$. Therefore for any $m'\in \C(m)$, a $\Pois\big( \H_{\Vert \C(m)} \big)$-BGW tree with root of type $m'$ is distributed as a $\Pois(\H)$-BGW tree with root of type $m'$.

We now state the main result of this section, Theorem \ref{thm exact multD}, proved in Section \ref{section6.7}.

\begin{theorem}\label{thm exact multD}
	The following holds.
	\begin{enumerate}
		\item $E$ is convex, closed and if $0 \preceq u \in E$ then for any $0\preceq v \preceq u$ we have $v\in E$, equivalently, $E = \CE(E)$.	
		\item We have $\mathring{E}^* \neq \varnothing \iff \spr(\H)<1$. Thus if $\spr(\H)<1$, the set $\mathring{E}$ is the interior of $E$ for the induced topology of $\R^{\M}$ on $\R^{\M}_+$ and $\delta E$ is the boundary of $E$ for the same topology.
		\item We have $\spr\big( \H_{\Vert \C(m)} \big) <1 \iff \exists \varepsilon >0, \ \varepsilon \e_m \in E$ and thus the following holds, \[E = \{0\} \iff \forall m\in\M, \ \spr\big(\H_{\Vert \C(m)} \big) \geq 1.\]
		\item The map $u\in E \mapsto \L(u)$ is injective, continuous on $\mathring{E}$ and on every $\CE(\{u\})$ with $u\in E$. If $\spr(\H)<1$ the application $u\in \mathring{E} \mapsto \L(u)$ is infinitely differentiable.
		\item For any $u\in\pos$, there exists at most one solution of the equation $x = u+\H(e^x-\ones)$ with $x\in\pos$ and $\spr(\H \diag(e^{x}))\leq 1$. Consequently, if such a solution exists, we must have $x = \L(u)$.
		\item If $\spr(\H)<1$ we have 
	\begin{align*}
			\mathring{E} & = \big\{ u \in\pos \mid \exists y\in\pos, \ u = y-\H(e^y-\ones), \ \spr\big(\H\diag(e^{y})\big) <1 \big\}\\
			& = \big\{ u \in E \mid \spr\big(\H\diag(e^{\L(u)})\big) <1 \big\}.
	\end{align*}
		\item If $\spr(\H)<1$ we have 
	\begin{align*}
	\delta E & = \big\{ u \in\pos \mid \exists y\in\pos, \ u = y-\H(e^y-\ones), \ \spr\big(\H\diag(e^{y})\big) = 1 \big\}\\
	& = \big\{ u \in E \mid \spr\big(\H\diag(e^{\L(u)})\big) = 1 \big\}.
	\end{align*}
	\end{enumerate}
\end{theorem}

Figure \ref{fig:fig1} gives a schematic representation of the topological properties, derived in Theorem \ref{thm exact multD}, of the set $E$ when $\spr(\H)< 1$.

\begin{figure}[h]
	\begin{tikzpicture}
		
		\definecolor{aqua}{rgb}{0.0,1.0,1.0}
		
		\draw[line width=1.5pt, -stealth] (-0.5,0)--(5,0);
		\draw[line width=1.5pt, -stealth] (0,-0.5)--(0,5);
		
		\fill[aqua, domain=0:60] (0,0) -- plot ({4*cos(90-\x)}, {4*sin(90-\x)}) -- ({2*sqrt(3)},0) -- cycle;
		
		\draw[red, line width=2pt, domain=0:60] (0,0) -- plot ({4*cos(90-\x)}, {4*sin(90-\x)}) -- ({2*sqrt(3)},0) -- cycle;
		
		\draw[blue, line width=2pt] (0,4)--(0,0)--({2*sqrt(3)},0);
		
		\begin{scope}
			\clip (3,{0-1pt}) rectangle (5,2);
			\draw[red, line width=2pt] ({2*sqrt(3)},1) -- ({2*sqrt(3)},-1);
		\end{scope}
		
		\begin{scope}
			\clip ({0-1pt},3) rectangle (1,5);
			\draw[red, line width=2pt, domain=-10:60] plot ({4*cos(90-\x)}, {4*sin(90-\x)});
		\end{scope}
		
		\draw (-0.5,-0.5) node {$0$};
		\draw (5,-0.5) node {$\R_+$};
		\draw (-0.5,5) node {$\R_+$};
		
		\draw ({4*cos(50)+0.5}, {4*sin(50)}) node {$u_c$};
		
		\draw[densely dotted] (0,0)--({4.5*cos(50)}, {4.5*sin(50)});
		
		\tikzstyle{every node}=[circle, draw, line width=0.8pt, fill=red, inner sep=0pt, minimum width=6pt]
		\draw ({4*cos(50)}, {4*sin(50)}) node {};
		
	\end{tikzpicture}
	\caption{
Illustration of the topological structure of $E$ when $\spr(\H)<1$. The boundary $\delta E$ is represented in red, while $\mathring{E}$ corresponds to the union of the blue and light blue regions. The light blue region is $\mathring{E}^*$. The point $u_c$ is critical in the direction $u_c/\Vert u_c\Vert$, meaning that it is the largest value in this direction for which $\L(u)$ remains finite.}
	\label{fig:fig1}
\end{figure}

\begin{remark}
	Points $(6)$ and $(7)$ allow to construct exact solutions for the exponential moments, and even maximal ones with the condition $\spr(\H\diag(e^y))=1$, by guessing values $y$ for $\L(u)$. However, one does not control the value and the direction obtained for the associated $u$.
\end{remark}

\begin{remark}\label{rmq5.7}
	Most of the above results are stated under the assumption $\spr(\H)<1$. If it is not the case then by $(3)$ there are two possibilities. Either $E = \{0\}$, or $E \neq \{0\}$ and there exists $m\in\M$ such that $\spr(\H_{\Vert \C(m)})<1$. In this second case consider the set of types $\tilde{\M}\subset \M$ defined by \[ \tilde{\M} = \{ m\in\M \mid \spr\big(\H_{\Vert \C(m)}\big)<1\}.\] From Remark \ref{rmk7.8} it is clear that \[\tilde{\M} = \bigcup_{m\in\tilde{\M}} \C(m),\] and thus for any $x$ with support in $\tilde{\M}$, then $\H x$ has also support in $\tilde{\M}$ and we have \[(\H x)_{\vert \tilde{\M}} = \H_{\Vert \tilde{\M}} x_{\vert \tilde{\M}}.\] Again by $(3)$ it is clear that $E$ lies entirely in the set \[\{u \in\pos \mid \forall m\in \M\setminus \tilde{\M}, \ u_m =0\}.\]
	Thus one can perform the study of $E$ as a subset of $\R_+^{\tilde{\M}}$ with $(\tilde{\M},\H_{\Vert \tilde{\M}})$ instead of $(\M,\H)$. By definition of $\tilde{\M}$ it is clear that there exists $\varepsilon>0$ such that for any $m \in \tilde{\M}$ we have $\varepsilon \e_m \in E$. Thus by convexity, up to reducing $\varepsilon$, we have $(\varepsilon)_{m\in\tilde{\M}} \in E$ and thus by point $(2)$ we have $\spr\big(\H_{\Vert\tilde{\M}}\big)<1$ and all the results applies for $(\tilde{\M},\H_{\Vert \tilde{\M}})$.
\end{remark}

\subsection{Proof of Theorem \ref{thm exact multD}}\label{section6.7}

$\bm{(1):}$ By \eqref{th1eq4} of Theorem \ref{theorem1} it is clear that $E$ is convex and that $E = \CE(E)\cap\R_+^{\M}$. Let now $u_n \in E$ such that $u_n \to u^*\in\pos$. Thus we have 
	\begin{equation*}
		u_n = \L(u_n) - \H(e^{\L(u_n)}-\ones) \xrightarrow[n\to\infty]{} u^*.
	\end{equation*}
	We first show that $(\L(u_n))_n$ is bounded. Suppose otherwise. Up to extraction by $\phi$, there exists a nonempty subset $S\subset\M$ such that $\L(u_{\phi(n)})_i\to\infty$ for all $i\in S$, while $\L(u_{\phi(n)})_i$ remains bounded by a constant $B$ for $i\notin S$.
	By further extracting we can suppose that there is $i_0$ in $S$ such that for all $n$,
	\[ \min_{i\in S} \L(u_{\phi(n)})_i = \L(u_{\phi(n)})_{i_0}.\]
	Then, either $H_{i_0,j} = 0$ for all $j\in S$ and we have 
	\[u^*_{i_0} \geq \lim_{n\to\infty} \Big(\L(u_{\phi(n)})_{i_0} - \sum_{j\in\M\setminus S} H_{i_0,j} (e^B-1)\Big) = \infty.\]
	Or there exists $j\in S$ such that $H_{i_0,j} >0$ and in this case we have
	\[u^*_{i_0} \leq \lim_{n\to\infty} \Big(\L(u_{\phi(n)})_{i_0} - H_{i_0,j}e^{\L(u_{\phi(n)})_{i_0}} + \H\ones \Big)  = -\infty.\] 
	In both cases we have a contradiction with $u^*\in\pos$, it follows that $\L(u_n)$ is bounded. Then, up to extraction we can assume that
	\[\L(u_n) \xrightarrow[n\to\infty]{} y.\]
	By continuity we have $y = u^* + \H(e^{y}-\ones)$, and $y\in\pos$ as a limit of a nonnegative and bounded sequence. Thus, by Theorem \ref{theorem1}, $u^*\in E$ and we conclude that $E$ is closed.
	
$\bm{(2):}$ Lemma \ref{lemma3.5} gives the direct implication and, by Theorem \ref{theorem1}, the reciprocal is also clear. Let us prove that $\mathring{E}$ is indeed the interior of $E$ if $\spr(\H)<1$. Suppose that $\spr(\H)<1$ and let $u\in\mathring{E}$. Then there exists $\varepsilon >0$ such that $\varepsilon \ones$ and $(1+\varepsilon)u$ are in $E$. By convexity we have \[u + \frac{\varepsilon^2}{1+\varepsilon}\ones \in E.\] Thus the $\R_+^{\M}$-ball of radius $\frac{\varepsilon^2}{1+\varepsilon}$, centered at $u$ for the norm $\vert \cdot \vert_{\infty}$ is contained in $E$. Reciprocally, if $u\in \delta E$, then $(1+\varepsilon)u \to u$ as $\varepsilon \searrow 0$ and $(1+\varepsilon)u\notin E$ by definition of $\delta E$. It follows that $u$ is not in the interior of $E$, which concludes since $E$ is closed in $\R^{\M}$ (and thus in $\R_+^{\M}$).
	
$\bm{(3):}$ Let $m\in\M$, and suppose that $\spr(\H_{\Vert \C(m)})<1$. It follows from Lemma \ref{lemma5.5} and point $\bm{(2)}$ that $\varepsilon \e_m \in E$ for $\varepsilon>0$ small enough.\\
	Reciprocally, suppose that $\spr(\H_{\Vert \C(m)})\geq 1$. Let $m\in\M$, and $\varepsilon>0$. By Lemma \ref{lemma3.5} we have
	\[\exp(\L(\varepsilon \e_m) \succeq \varepsilon \sum_{n\geq 0} \H^n \e_m.\]
	Denote $z = (\e_m)_{\vert \C(m)}$ and $A = \H_{\Vert \C(m)}$, thus from the support considerations of Lemma \ref{lemma5.5} and Remark \ref{rmk7.8} we have,
	\[ \exp(\L(\varepsilon\e_m )_{\vert\C(m)}) \succeq \varepsilon \sum_{n\geq 0} A^n z.\] 
	Let $\ones_{\C(m)} = (1,\cdots,1)\in\R^{\C(m)}$. By definition of $\C(m)$, there exist $N\geq 1$ and $\eta >0$ such that we have $\sum_{k = 0}^{N-1} A^k z \succeq \eta \ones_{\C(m)} \succ 0$ in $\R^{\C(m)}$. Thus we have,
	\begin{align*}
		\exp(\L(\varepsilon\e_m )_{\vert\C(m)}) & \succeq \varepsilon N^{-1} \sum_{n\geq 0} A^n \sum_{k= 0}^{N-1} A^k z \\
		& \succeq \varepsilon \eta N^{-1} \sum_{n\geq 0} A^n \ones_{\C(m)}.
	\end{align*}
	Since $\rho = \spr(A)\geq 1$, the Weak Perron-Frobenius Theorem yields the existence of a nonzero $\kappa \succeq 0$ such that $A\kappa = \rho \kappa$. Rescaling $\kappa$ if necessary, we may assume that $\ones_{\C(m)} \succeq \kappa$. Thus we have 
	\begin{equation*}
		 \exp(\L(\varepsilon\e_m )_{\vert\C(m)}) \succeq \varepsilon \eta N^{-1} \sum_{n\geq 0} \rho^n \kappa  = \infty \kappa.
	\end{equation*}
	Which proves that $\varepsilon \e_m \notin E$, where $\varepsilon >0$ is arbitrary, which concludes.
	
	Let us now prove the second equivalence. On one hand, if $E = \{0\}$ then for any $m\in\M$ and any $\varepsilon >0$ we have $\varepsilon \e_m \notin E$ and thus $\spr(\H_{\Vert \C(m)})\geq 1$. Reciprocally, for all $m\in\M$, for all $\varepsilon>0$ we have $\varepsilon \e_m \notin E$. Thus if $E \neq \{0\}$ we would have a contradiction since $u\neq 0$ in $E$ has at least one positive entry, say $m$, and thus $u_m \e_m \in E$ since $u_m \e_m \preceq u\in E$, which is a contradiction.
	
$\bm{(4):}$ Let $u,v\in E$ such that $\L(u) = \L(v)$. Then we have $u = \L(u)-\H(e^{\L(u)}-\ones) = v$ which proves that $\L:E\longrightarrow \pos$ is injective. For $F\subset \R^{\M}$ define the full lower cone of $F$ by $\CEt(F) = \{ u\in\R^{\M} \mid \exists v\in F, \ u\preceq v\}$. Clearly by dominated convergence one can extend $\L$ as follows,
\[\fonction{\L}{\CEt(E)}{\R^{\M}}{u}{\Big[\log\Big(\E{e^{u\trp \card_{\M}(\T^m)}}\Big)\Big]_{m\in\M}}.\] Continuity on $\CE(\{u\})$, and even on $\CEt(\{u\})$, is clear by dominated convergence, which also gives continuity on $\mathring{E}$ and $\CEt(\mathring{E})$. Suppose now that $\spr(\H)<1$. Thus $\CEt(\mathring{E})$ is an open set of $\R^{\M}$. For any $u\in \CEt(\mathring{E})$ there exists a neighbourhood $V$ of $u$, $\varepsilon >0$ and $u_V \in \mathring{E}$ such that for all $v\in V$ we have $v \preceq u_V-\varepsilon\ones$. For every multi-index $\alpha$, derivatives of order $\alpha$ of $v\mapsto e^{v\trp \card_{\M}(\T^m)}$ are dominated on $V$ by $K_{\alpha,V,\varepsilon} e^{u_V\trp \card_{\M}(\T^m)}$ with some finite constant $K_{\alpha,V,\varepsilon}<\infty$. Thus from classical results on integral depending on a parameter, $\L$ is infinitely differentiable since one can apply dominated convergence theorems to all the derivatives.
	
$\bm{(5):}$ Suppose that there is two such solutions $y,z$. Recall by Theorem \ref{theorem1} that $\L(u)$ is finite if and only if there exists a finite nonnegative solution to the fixed-point equation and, in that case, $\L(u)$ is the smallest nonnegative solution. It follows that $u\in E$. We will show that $y=z=\L(u)$. Thus, without loss of generality we can suppose that $y=\L(u) \preceq z$. We have 
	\begin{align*}
		z - \L(u) & = \H\big( e^z-e^{\L(u)}\big) \\
		& \preceq \H \diag(e^z) (z-\L(u))\\
		& \preceq \limsup_{n\to\infty} \big[\H \diag(e^z)\big]^n (z-\L(u)).
	\end{align*}
	Thus if $\spr(\H \diag(e^z))<1$ we have $z = \L(u)$.\\
	If $\spr(\H \diag(e^z))=1$ then the argument is more subtle. Denote $x = z-\L(u) \succeq 0$. We have \[x = \H \diag(e^{\L(u)})(e^x-\ones).\] Let $S  \subset \M$ the support of $x$ and $A = \H \diag(e^{\L(u)})$. Because the support of $e^x-\ones$ is also $S$, it follows that for all $i\in S$ there is at least one $j\in S$ such that $A_{ij}>0$, and that $A_{ij}=0$ for any $i\notin S$ and $j\in S$. Let $\Phi$ the function defined by \[ \fonction{\Phi}{\R}{\R^{\M}}{t}{A(e^{tx}-\ones) - tx}.\] Observe that $\Phi(0)=\Phi(1)=0$. Each coordinate of $\Phi$ is a convex function (as nonnegative linear combination of convex functions). More precisely, for all $i\in S$ the $i$-th coordinate of $\Phi$ is strictly convex since there exists $j\in S$ such that $A_{ij} >0$ and $t\mapsto e^{tx_j}-1$ is strictly convex. Also, the $i$-th coordinate of $\Phi$ for $i\notin S$ is constant of value $0$ since $A_{ij} = 0$ for all $j\in\S$. Thus, since $\Phi(0) = \Phi(1) = 0$ we must have for some $\varepsilon > 0$ is small enough, \[\Phi'(1) = A \diag(e^{x})x-x \succeq \varepsilon \ones_{S}\] where $\ones_{S}$ is the vector with entries $1$ on elements of $S$ and $0$ outside. Up to reducing $\varepsilon$ we may assume that \[ A \diag(e^{x})x-x \succeq \varepsilon x.\] Thus we have \[ x \preceq (1+\varepsilon)^{-1} A \diag(e^{x})x.\] One can iterate and obtains \[0\preceq x \preceq \Big((1+\varepsilon)^{-1} A \diag(e^{x})\Big)^n x, \quad \forall n\geq 0.\] Recall that $A \diag(e^{x}) = \H \diag(e^z)$ is assumed to have spectral radius equal to one. Thus $x=0$ by letting $n\to\infty$, i.e. $z = \L(u)$.
	
$\bm{(6):}$ First remark that both set on the right hand side are equal by $\bm{(5)}$. Let $u\in \mathring{E}$. The fixed-point equation is not specific to $u\succeq 0$. If $v\in \CEt(\mathring{E})$, possibly not nonnegative, $\L(v)$ is well-defined as explained in the proof of $\bm{(4)}$. The same proof as for Theorem \ref{theorem1}, based on the branching property, shows that \[ \L(v) = v + \H\big( e^{\L(v)}-\ones\big).\] The proof of $\bm{(4)}$ also shows that $\L$, extended to the open (for $\R^{\M}$) set $\CEt(\mathring{E})$, is differentiable. One can differentiate the equality $v = \L(v) - \H\big( e^{\L(v)}-\ones\big)$ at $u\in \mathring{E} \subset \CEt(\mathring{E})$ and obtains 
	\begin{equation}
		\Id = \big(\Id - \H\diag(e^{\L(u)})\big) D_u\L. 
	\end{equation}
	Thus, $\big(\Id - \H\diag(e^{\L(u)})\big)^{-1}$ exists. Suppose that $\rho := \spr(\H\diag(e^{\L(u)}))= 1$. Then since $\H\diag(e^{\L(u)}) \succeq 0$ by the Weak Perron-Frobenius theorem, $1$ is eigenvalue, which is a contradiction. Suppose then that $\rho >1$. Since $\spr(\H)<1$, by continuity, there exists $0 \preceq  v\preceq u$ such that $\spr(\H\diag(e^{\L(v)}))=1$ which is, as we shown in the first case, a contradiction since $v\in \mathring{E}$. Thus $\rho <1$, which proves the inclusion "$ \subset$".\\
	Let $u\in E$ such that $\spr(\H \diag(e^{\L(u)}))<1$. Let $\varepsilon >0$, $v = u+\varepsilon \ones$, and define $z_n$ by \[\forall n\in\N, \ z_n = \psi^n_v(0)-\psi^n_u(0) \succeq 0,\] where $\psi_w(x) = w+\H(e^x-\ones)$. Then for $n\geq 0$ we have 
	\begin{align*} 
		z_{n+1} & = \varepsilon \ones + \H\big( e^{\psi^n_v(0)} - e^{\psi^n_u(0)}\big) \\
		& = \varepsilon \ones + \H\diag(e^{\psi^n_u(0)}) \big(e^{z_n} - \ones\big) \\
		& \preceq \varepsilon \ones + \H\diag(e^{\L(u)}) \big(e^{z_n} - \ones\big). \\
	\end{align*}
	Thus if $g_{\varepsilon\ones}(x) = \varepsilon\ones + \H\diag(e^{\L(u)}) (e^{x} - \ones)$, we have
	\begin{align*}
		z_{n+1} & \preceq g_{\varepsilon\ones}(z_n) \\
		& \preceq g_{\varepsilon\ones}^{n+1}(z_0)\\
		& = g_{\varepsilon\ones}^{n+1}(0)
	\end{align*}
	Since $\H\diag(e^{\L(u)})$ has spectral radius smaller than one, by $\bm{(2)}$ applied to $\H \leftarrow\H\diag(e^{\L(u)})$, for $\varepsilon$ small enough the sequence $\big(g_{\varepsilon\ones}^{n+1}(0)\big)_n$ converges to $\L_{\H\diag(e^{\L(u)})}(\varepsilon\ones)$ finite. It follows that $(z_n)_n$ is bounded and thus converges to a finite limit: indeed, $(\psi^n_u(0))_n$ converges (to $\L(u)$) and $(\psi_v^n(0))_n$ is increasing and now bounded thus converges to a finite limit. Which proves that $v \in E$ and thus $u\in \mathring{E}$.
	
$\bm{(7):}$ First remark that both set on the right hand side are equal by $\bm{(5)}$.\\
	Let $u\in\delta E$. Then for any $t<1$ we have $tu\in \mathring{E}$ and thus $\spr(\H\diag(e^{\L(tu)}))<1$. By continuity on $\CE(\{u\})$ we have $\spr(\H\diag(e^{\L(u)})) \leq 1$. Since $u\notin \mathring{E}$, by $\bm{(6)}$ we must have $\spr(\H\diag(e^{\L(u)})) = 1$.\\
	Let $u\in E$ such that $\spr(\H\diag(e^{\L(u)})) = 1$. Then by $\bm{(6)}$ it is clear that $u\notin \mathring{E}$, This completes the proof.

Theorem \ref{thm exact multD} is proved.

\appendix

\section{}\label{appendixA}

\subsection{Perron-Frobenius Theorems}\label{appendixA.1}

Let us recall the classical Perron-Frobenius Theorem and its weak version.

\textit{\textbf{Perron-Frobenius Theorem:}} Let $A$ be an irreducible $N\times N$ matrix with nonnegative entries. Let $\rho = \spr(A)$. Then there exists $\kappa \in (0,\infty)^N$ such that \[A\kappa = \rho \kappa.\]

We can remove the irreducible condition up to a weaker conclusion, which is the theorem stated below.

\textit{\textbf{Weak Perron-Frobenius Theorem:}} Let $A$ be a $N\times N$ matrix with nonnegative entries. Let $\rho = \spr(A)$. Then there exists $\kappa \in [0,\infty)^N$ non zero such that \[A\kappa = \rho \kappa.\]

To prove this weak version we only need to consider a sequence $A_{\varepsilon} \xrightarrow[\varepsilon\to 0]{} A$ such that $A_{\varepsilon}$ satisfies the conditions of the classical Perron-Frobenius Theorem. Take for example $A_{\varepsilon} = A + \varepsilon J$ with $J$ the matrix with only ones. By taking the limit (up to extraction) in \[A_{\varepsilon} \kappa_{\varepsilon} = \rho_{\varepsilon} \kappa_{\varepsilon}\] we obtain a leading eigenvector $\kappa \succeq 0$ and the spectral radius $\rho$ such that $A\kappa = \rho \kappa$.

\subsection{Technical results on bivariate convolution}\label{appendixA.2}

In this section we introduce some results about bivariate convolution which are key to prove Theorem \ref{theorem4.5}. The proofs are given in section \ref{proofs A}

Let $\bm{\delta_0} = \delta_0 \Id$ with $\delta_0$ the Dirac kernel at delay $0$, defined by $\delta_0(s,A)= \ind_{s\in A}$, so that $\bm{\delta_0}$ is the unit for the bivariate convolution $\star$.

For any matrix $B$ we define the entrywise absolute value of $B$ by $\vert B\vert = (\vert B_{ij}\vert)_{i,j}$.

We recall that bivariate convolution and related operators like residual intergal are introduced in Section \ref{sec4.1}.

This first result extends the classical $L^{\infty}$ and $L^1$ bounds of a convolution to bivariate convolution.

\begin{proposition}\label{prop4.1}
	Let $\bm{v}:\R^2 \to \mathcal{M}_{n,p}$ and $\bm{w} : \R^2 \to \mathcal{M}_{p,q}$ be bivariate matrix-valued functions. The following holds,
	\begin{itemize}
		\item For any $s\in\R$ we have \[\Vone{\bm{v}\star\bm{w}(s,\cdot)} \preceq \Vone{\bm{v}(s,\cdot)}  \Vityone{\bm{w}}.\] Thus we have \[ \Vityone{\bm{v}\star\bm{w}} \preceq \Vityone{\bm{v}}  \Vityone{\bm{w}}.\]
		\item For any $s,t\in\R$ we have \[\vert\bm{v}\star\bm{w}(s,t)\vert \preceq \Vone{\bm{v}(s,\cdot)}  \Vity{ \bm{w}(\cdot,t)}.\] Thus we have \[ \Vityty{\bm{v}\star\bm{w}} \preceq \Vityone{\bm{v}} \Vityty{\bm{w}}.\]
	\end{itemize}
\end{proposition}

\begin{remark}
	If $\bm{v}(s,t) = f(t-s)$ and $\bm{w}(s,t)=g(t-s)$ then this result is just the classical result $\Vone{f\star g} \leq \Vone{f} \Vone{g}$ and $\Vity{f\star g} \leq \Vone{f} \Vity{g}$.
\end{remark}

The following Proposition links the decay of $\bm{v}$ and $\bm{w}$ to the decay of $\bm{v}\star\bm{w}$.

\begin{proposition}\label{prop4.2.0}
	Let $\bm{v}:\R^2 \to \mathcal{M}_{n,\ell}$ and $\bm{w} : \R^2 \to \mathcal{M}_{\ell,r}$ be nonnegative (entrywise) bivariate matrix-valued functions. Then for any $d\geq 0$ and any $p\in[0,1]$ the following holds
	\begin{equation}\label{eq4.1}
	\Res^{\infty}_{\bm{v}\star\bm{w}}(d) \preceq \Vityone{\bm{v}}  \Res^{\infty}_{\bm{w}}(qd) + \Res^{\infty}_{\bm{v}}(pd)  \Vityone{\bm{w}}
\end{equation}
with $q=1-p$.
\end{proposition}

\begin{remark}
	This result extends to functions $\bm{v}$ and $\bm{w}$ not necessarily nonnegative by replacing $\Res^{\infty}_{\bm{v}}$ and $\Res^{\infty}_{\bm{w}}$ by $\Res^{\infty}_{\vert\bm{v}\vert}$ and $\Res^{\infty}_{\vert\bm{w}\vert}$.
\end{remark}

This result is the starting point for the following proposition that links the decay of a nonnegative bivariate matrix-valued function $\bm{v}$ to the decay of $\bm{\psi} = \sum_{n\geq 1} \bm{v}^{\star n}$.

\begin{proposition}\label{prop4.2}
	Let $\bm{v}$ be a nonnegative bivariate matrix-valued function. Suppose that $\bm{V} =\Vityone{\bm{v}}$ has spectral radius smaller than one, then the function $\bm{\psi} = \sum_{n\geq 1} \bm{v}^{\star n}$ is well defined and satisfies for all $d>0$ and all $p\in(0,1)$,
	\[\Res_{\bm{\psi}}^{\infty}(d) \preceq \sum_{k=0}^{\infty} \bm{V}^k \Res_{\bm{v}}^{\infty}\big((1-p)p^k d\big) (\Id-\bm{V})^{-1}. \]  
\end{proposition}

\begin{remark}
	Once again this result applies for $\bm{v}$ possibly not nonnegative by using $\vert \bm{v}\vert$ instead of $\bm{v}$ in the upper bound.
\end{remark}

The following result establishes that $\bm{\psi}$ exhibits polynomial decay whenever $\bm{v}$ does.

\begin{proposition}\label{prop4.3}
	Let $\bm{v}$ be a nonnegative bivariate matrix-valued functions. Suppose that $\bm{V} =\Vityone{\bm{v}} \in\Ge{\alpha}{K}$ with $\alpha<1$ and that
	\[\Res_{\bm{v}}^{\infty}(t) \preceq \frac{A}{(1+t)^{\gamma}} \bm{V}, \quad \forall t\geq 0.\]
	Then the function $\bm{\psi} = \sum_{n\geq 1} \bm{v}^{\star n}$ is well defined and satisfies, for all $t> 0$ and all $0<\delta <1$,
	\begin{equation*}
	\Res_{\bm{\psi}}^{\infty}(t) \preceq \frac{A \big(1+\frac{\gamma}{\delta \log(1/\alpha)}\big)^{\gamma}}{(1+t)^{\gamma(1-\delta)}}\bm{V} (\Id-\bm{V})^{-2} + \bm{V}^{\lfloor\frac{\gamma(1-\delta) \log(1+t)}{\log(1/\alpha)}\rfloor+2} (\Id-\bm{V})^{-2}.
	\end{equation*}
	In particular,
	\begin{equation*}
		\Vvert \Res_{\bm{\psi}}^{\infty}(t) \Vvert_{\infty} \leq \frac{\alpha K(K+1) \Big[ A\big(1+\frac{\gamma}{\delta \log(1/\alpha)}\big)^{\gamma} + 1\Big]}{(1-\alpha)^2(1+t)^{\gamma(1-\delta)}}, \quad \forall t>0.
	\end{equation*}
\end{proposition}

Similarly to the polynomial case, the following result establishes that $\bm{\psi}$ exhibits exponential decay whenever $\bm{v}$ does.

\begin{proposition}\label{prop4.4}
	Let $\bm{v}$ be a nonnegative bivariate matrix-valued functions. Suppose that $\bm{V} =\Vityone{\bm{v}} \in\Ge{\alpha}{K}$ with $\alpha<1$ and that there exist $A\geq 0$ and $c>0$ such that
	\[\Res_{\bm{v}}^{\infty}(t) \preceq A e^{-ct} \bm{V}, \quad \forall t\geq 0.\]
	Then the function $\bm{\psi} = \sum_{n\geq 1} \bm{v}^{\star n}$ is well defined and satisfies for all $t>0$ and all $0<\varepsilon< \frac{c(1-\alpha)}{(1-\alpha)+A\alpha}$, \[ \Res_{\bm{\psi}}^{\infty}(t) \preceq  e^{-\varepsilon t} \bm{V}_{\varepsilon}(\Id-\bm{V}_{\varepsilon})^{-1} \] with \[\bm{V}_{\varepsilon} \preceq \alpha \big(1 + \frac{A\varepsilon}{c-\varepsilon}\big) \bm{V}.\] In particular,
	\[\Vvert \Res_{\bm{\psi}}^{\infty}(t) \Vvert_{\infty} \leq \frac{ K \alpha \big(1 + \frac{A\varepsilon}{c-\varepsilon}\big)
	}{1-\alpha \big(1 + \frac{A\varepsilon}{c-\varepsilon}\big)} e^{-\varepsilon t}, \quad \forall t>0.\]
\end{proposition}

\begin{remark}
	In the one dimensional case, if one take $v(s,t) = a e^{-b(t-s)}\ind_{s\leq t}$ with $b>a >0$, then in the proof of Lemma 9 of \cite{tchouanti_separation_2024}, the authors have shown that \[\bm{\psi}(s,t) = a e^{-(b-a)(t-s)}\ind_{s\leq t},\] which proves that \[ \Res_{\bm{\psi}}^{\infty}(t) \lesssim e^{-(b-a)t}.\] This choice for $v$ leads to $\Res_v^{\infty}(t)= \frac{a}{b}e^{-bt}= 1\times e^{-bt} V$ with $V=a/b \in\Ge{a/b}{1}$. Thus Proposition \ref{prop4.4} gives the following limit value for $\varepsilon$, \[\frac{b(1-a/b)}{(1-a/b) + 1\times a/b} = b-a,\]
	which is the exponential decay rate of $\Res_{\bm{\psi}}^{\infty}$.
\end{remark}

\subsection{Proof of Appendix \ref{appendixA.2} results}

\begin{proof}[Proof of Proposition \ref{prop4.1}]
	Let $s\in\R$, then
	\begin{align*}
		\itg \vert \bm{v}\star\bm{w}(s,t) \vert dt  = \itg \Big\vert\itg \bm{v}(s,x) \bm{w}(x,t)dx \Big\vert dt & \preceq \itg \itg \vert \bm{v}(s,x) \vert  \vert \bm{w}(x,t) \vert dt dx \\
		& \preceq \itg \Big(\vert \bm{v}(s,x) \vert \itg   \vert \bm{w}(x,t) \vert dt\Big) dx \\
		& \preceq \itg \vert \bm{v}(s,x) \vert dx \Vityone{
		\bm{w}} \\
		& \preceq \Vityone{\bm{v}}\Vityone{\bm{w}}.
	\end{align*}
	Similarly, for any $s,t\in\R$ we have
	\begin{align*}
		\vert \bm{v}\star\bm{w}(s,t) \vert  =  \Big\vert\itg \bm{v}(s,x) \bm{w}(x,t)dx \Big\vert & \preceq \itg \vert \bm{v}(s,x)\vert \vert \bm{w}(x,t)\vert dx\\
		& \preceq \itg \vert \bm{v}(s,x)\vert  dx \Vity{\bm{w}(\cdot,t)}\\
		& \preceq \Vityone{\bm{v}} \Vityty{\bm{w}}.\\
	\end{align*}
\end{proof}

\begin{proof}[Proof of Proposition \ref{prop4.2.0}]
	Let us look at the residual integral of $\bm{v}\star\bm{w}$. For any $s\in\R$ and $d\geq 0$ we have
\begin{align*}
	\Res_{\bm{v}\star\bm{w}}(s,s+d)  = \itg_{s+d}^{\infty} \itg \bm{v}(s,x) \bm{w}(x,t)dxdt & = \itg \itg_{s+d}^{\infty}  \bm{v}(s,x) \bm{w}(x,t)dtdx \\
	& = \itg \bm{v}(s,x) \Res_{\bm{w}}(x,s+d)dx.
\end{align*}
By splitting the integral at $s+pd$ with $p\in [0,1]$ we have
\begin{align*}
	\Res_{\bm{v}\star\bm{w}}(s,s+d) & = \itg_{-\infty}^{s+pd} \bm{v}(s,x) \Res_{\bm{w}}(x,s+d)dx + \itg_{s+pd}^{\infty} \bm{v}(s,x) \Res_{\bm{w}}(x,s+d)dx\\
	& \preceq \itg_{-\infty}^{s+pd} \bm{v}(s,x) \Res_{\bm{w}}^{\infty}(s+d-x)dx + \itg_{s+pd}^{\infty} \bm{v}(s,x) \Vityty{\Res_{\bm{w}}}dx\\
	& \preceq \itg_{-\infty}^{s+pd} \bm{v}(s,x) dx \Res_{\bm{w}}^{\infty}(d-pd) + \itg_{s+pd}^{\infty} \bm{v}(s,x) dx \Vityone{\bm{w}}\\
	& \preceq \Vityone{\bm{v}} \Res_{\bm{w}}^{\infty}(qd) + \Res_{\bm{v}}(s,s+pd) \Vityone{\bm{w}}.
\end{align*}
Taking the $L^{\infty}$ norm on $s$ on both side leads to the result.
\end{proof}

\begin{proof}[Proof of Proposition \ref{prop4.2}]\label{proofs A}
	The series defining $\bm{\psi}$ converges since
	\[\Big\Vert\sum_{n\geq 1} \bm{v}^{\star n}\Big\Vert_{L^{\infty}L^1} \preceq \sum_{n\geq 1} \bm{V}^n \preceq \bm{V}(\Id-\bm{V})^{-1}.\] 
	Recall that $\bm{\delta_0} = \delta_0 \Id $ is the unit for $\star$. By definition we have $\bm{\psi} = \bm{v} \star (\bm{\psi} + \bm{\delta_0})$. Since $\Res_{\bm{\delta_0}}^{\infty}(t) = 0$ for $t>0$, applying Proposition \ref{prop4.2.0} gives
	\[\Res^{\infty}_{\bm{\psi}}(d) \preceq \bm{V} \Res^{\infty}_{\bm{\psi}}(pd) + \Res^{\infty}_{\bm{v}}((1-p)d)(\Id-\bm{V})^{-1}, \quad d>0, \ 0<p<1. \]
	Iterating $n$ times leads to
	\[\Res^{\infty}_{\bm{\psi}}(d) \preceq \sum_{k=0}^{n-1} \bm{V}^k \Res^{\infty}_{\bm{v}}((1-p)p^kd)(\Id-\bm{V})^{-1} + \bm{V}^{n} \Res^{\infty}_{\bm{\psi}}(p^n d), \quad d>0, \ 0<p<1. \]
	Since $\Res_{\bm{\psi}}^{\infty}(t) \preceq \bm{V}(\Id-\bm{V})^{-1}$ for any $t\geq 0$ and since $\spr(\bm{V})<1$, taking $n\to\infty$ yields,
	\[\Res^{\infty}_{\bm{\psi}}(d) \preceq \sum_{k=0}^{\infty} \bm{V}^k \Res^{\infty}_{\bm{v}}((1-p)p^kd)(\Id-\bm{V})^{-1}, \quad d>0, \ 0<p<1. \]
\end{proof}

\begin{proof}[Proof of Proposition \ref{prop4.3}]
	Let $t > 0$. We use Proposition \ref{prop4.2} with $p = \exp(-\delta / \varepsilon)$ and cut the sum at $k_0 = \lfloor\varepsilon \log(1+ t)\rfloor \geq 0$ with \[\varepsilon = \frac{\gamma(1-\delta)}{\log(1/\alpha)}, \quad \text{and thus} \quad p = \exp(-\delta/\varepsilon) = \exp\Big(-\frac{\delta \log(1/\alpha)}{\gamma (1-\delta)} \Big). \]
	Thus we have
	\begin{align*}
		\Res^{\infty}_{\bm{\psi}}(t) & \preceq \sum_{k=0}^{k_0} \bm{V}^k \Res^{\infty}_{\bm{v}}((1-p)p^k t)(\Id-\bm{V})^{-1} + \sum_{k=k_0+1}^{\infty} \bm{V}^k \Res^{\infty}_{\bm{v}}((1-p)p^k t)(\Id-\bm{V})^{-1} \\
		& \preceq  (\Id-\bm{V})^{-1} \Res_{\bm{v}}((1-p)p^{k_0} t)(\Id-\bm{V})^{-1} + (\Id-\bm{V})^{-1} \bm{V}^{k_0+2} (\Id-\bm{V})^{-1}. \\
	\end{align*}
	By definition we have
	\[ (1-p)p^{k_0} t\geq (1-p)p^{\varepsilon\log(1+ t)} t = ( 1 - p) (1+t)^{-\delta} t.\]
	Thus we have
	\begin{align*}
		\Res^{\infty}_{\bm{v}}((1-p)p^{k_0} t) & \preceq \frac{A}{\Big[ 1+ ( 1 -p) (1+t)^{-\delta}t\Big]^{\gamma}} \bm{V} \\
		& \preceq \frac{A (1-p)^{-\gamma}}{( 1+ t)^{\gamma(1-\delta)}} \bm{V}
	\end{align*}
	where we used the elementary inequality \[\frac{1}{1+c(1+x)^{-a}x} \leq  \frac{c^{-1}}{(1+x)^{1-a}}, \quad x\geq 0, \ a,c\in (0,1). \]
	For $x>0$ we have $(1-e^{-x})^{-1}\leq 1+x^{-1}$ thus,
	\begin{align*}
		(1-p)^{-\gamma}  = \Big(1 - e^{-\frac{\delta \log(1/\alpha)}{\gamma(1-\delta)}}\Big)^{-\gamma}  \leq \Big(1 + \frac{\gamma(1-\delta)}{\delta \log(1/\alpha)}\Big)^{\gamma} & \leq \Big(1 + \frac{\gamma}{\delta \log(1/\alpha)}\Big)^{\gamma}.
	\end{align*}
	It proves the first result since $\bm{V}$ commutes with $(\Id-\bm{V})^{-1}$. Let us now check the second bound on the norm. Since we have for $j\geq 1$,
	\[ \Vvert \bm{V}^j(\Id-\bm{V})^{-1} \Vvert_{\infty} \leq \sum_{i\geq j} K\alpha^i = \frac{K \alpha^j}{1-\alpha},\]
	and that
	\[\Vvert (\Id-\bm{V})^{-1} \Vvert_{\infty} \leq 1+\sum_{i\geq 1} K\alpha^i \leq \frac{K +1}{1-\alpha},\]
	it is clear that 
	\begin{align*}
		\Vvert \Res^{\infty}_{\bm{\psi}}(t) \Vvert_{\infty}  & \leq \frac{K(K+1) \alpha}{(1-\alpha)^2} \times \frac{A \big(1 + \frac{\gamma}{\delta \log(1/\alpha)}\big)^{\gamma}}{( 1+ t)^{\gamma(1-\delta)}} + \frac{K(K+1) \alpha^{k_0+2}}{(1-\alpha)^2}.
	\end{align*} 
	Since we have
	\[ \alpha^{k_0+2} \leq \alpha \alpha^{\varepsilon \log(1+t)} = \alpha (1+t)^{-\gamma(1-\delta)},\]
	the bound is clear.
\end{proof}

\begin{proof}[Proof of Proposition \ref{prop4.4}]
	For $0< \varepsilon < \frac{c(1-\alpha)}{(1-\alpha)+ A\alpha }$ consider $\bm{v}_{\varepsilon}(s,t) = \bm{v}(s,t) e^{\varepsilon(t-s)}$. We can express $\psi$ as follows,
	\[\bm{\psi}(s,t) = e^{-\varepsilon(t-s)} \sum_{n\geq 1} \bm{v}_{\varepsilon}^{\star n} (s,t) = e^{-\varepsilon(t-s)} \bm{\psi}_{\varepsilon}(s,t) .\]
	Thus we have
	\[ \Res_{\bm{\psi}}^{\infty}(t) \preceq e^{-\varepsilon t } \Res_{\bm{\psi}_{\varepsilon}}^{\infty}(t).\]
	Then, for any $t$ we have $\Res_{\bm{\psi}_{\varepsilon}}^{\infty}(t) \preceq \bm{V}_{\varepsilon}(\Id-\bm{V}_{\varepsilon})^{-1}$ where
	\[\bm{V}_{\varepsilon} = \Vityone{\bm{v}_{\varepsilon}}.\]
	Let $s\in\R$, we have
	\begin{align*}
		\itg \bm{v}(s,t)e^{\varepsilon(t-s)}dt & = \Big[ -\Res_{\bm{v}}(s,t) e^{\varepsilon(t-s)}\Big]_{t=-\infty}^{t=\infty} + \varepsilon \itg \Res_{\bm{v}}(s,t) e^{\varepsilon(t-s)} dt\\
		& \preceq 0 + \varepsilon \itg_{-\infty}^s \Res_{\bm{v}}(s,t) e^{\varepsilon(t-s)} dt + \varepsilon \itg_s^{\infty} \Res_{\bm{v}}(s,t) e^{\varepsilon(t-s)} dt\\
		& \preceq \varepsilon \itg_{-\infty}^0 e^{\varepsilon t}dt \cdot \bm{V} + A \varepsilon \itg_0^{\infty} e^{-(c-\varepsilon)t} dt \cdot \bm{V}\\
		& =  \Big(1 + \frac{A\varepsilon}{c-\varepsilon}\Big)\bm{V}.
	\end{align*}
	Thus $\bm{V}_{\varepsilon} \preceq \big(1 + \frac{A\varepsilon}{c-\varepsilon}\big)\bm{V}$. By the assumption on $\varepsilon$, we have, $\alpha\big(1+\frac{A\varepsilon}{c-\varepsilon}\big)<1$, hence $\spr(\bm V_\varepsilon)<1$. In fact we have $\bm{V}_{\varepsilon} \in \Ge{\alpha \big(1 + \frac{A\varepsilon}{c-\varepsilon}\big)}{K}$, from which the bound
	\[\Vvert \Res_{\bm{\psi}}^{\infty}(t) \Vvert_{\infty} \leq \frac{ K \alpha \big(1 + \frac{A\varepsilon}{c-\varepsilon}\big)
	}{1-\alpha \big(1 + \frac{A\varepsilon}{c-\varepsilon}\big)} e^{-\varepsilon t}, \quad t>0\]
	follows, which concludes the proof.
\end{proof}

\section*{Acknowledgements}
I would like to thank my PhD advisors, Vincent Rivoirard and Patricia Reynaud-Bouret for their advices and helpful comments. I also thank the anonymous Referees, the Editor and Associate Editor for their careful reading and valuable suggestions, which helped improve the manuscript.

\bibliographystyle{alea3}
\bibliography{bibfile}

@article{bremaud_stability_1996,
	title = {Stability of {Nonlinear} {Hawkes} {Processes}},
	volume = {24},
	issn = {0091-1798},
	number = {3},
	journal = {The Annals of Probability},
	publisher = {Institute of Mathematical Statistics},
	author = {Brémaud, Pierre and Massoulié, Laurent},
	year = {1996},
	pages = {1563--1588},
}

@article{bacry_limit_2013,
	series = {A {Special} {Issue} on the {Occasion} of the 2013 {International} {Year} of {Statistics}},
	title = {Some limit theorems for {Hawkes} processes and application to financial statistics},
	volume = {123},
	issn = {0304-4149},
	doi = {10.1016/j.spa.2013.04.007},
	number = {7},
	journal = {Stochastic Processes and their Applications},
	author = {Bacry, E. and Delattre, S. and Hoffmann, M. and Muzy, J. F.},
	month = jul,
	year = {2013},
	pages = {2475--2499},
}

@article{hansen_lasso_2015,
	title = {Lasso and probabilistic inequalities for multivariate point processes},
	volume = {21},
	issn = {1350-7265},
	number = {1},
	journal = {Bernoulli},
	publisher = {[Bernoulli Society for Mathematical Statistics and Probability, International Statistical Institute (ISI)]},
	author = {Hansen, Niels Richard and Reynaud-Bouret, Patricia and Rivoirard, Vincent},
	year = {2015},
	pages = {83--143},
}

@article{hawkes_cluster_1974,
	title = {A {Cluster} {Process} {Representation} of a {Self}-{Exciting} {Process}},
	volume = {11},
	issn = {0021-9002},
	doi = {10.2307/3212693},
	number = {3},
	journal = {Journal of Applied Probability},
	publisher = {Applied Probability Trust},
	author = {Hawkes, Alan G. and Oakes, David},
	year = {1974},
	pages = {493--503},
}

@article{hawkes_spectra_1971,
	title = {Spectra of {Some} {Self}-{Exciting} and {Mutually} {Exciting} {Point} {Processes}},
	volume = {58},
	issn = {0006-3444},
	doi = {10.2307/2334319},
	number = {1},
	journal = {Biometrika},
	publisher = {[Oxford University Press, Biometrika Trust]},
	author = {Hawkes, Alan G.},
	year = {1971},
	pages = {83--90},
}

@article{ogata_statistical_1988,
	title = {Statistical {Models} for {Earthquake} {Occurrences} and {Residual} {Analysis} for {Point} {Processes}},
	volume = {83},
	issn = {0162-1459},
	doi = {10.2307/2288914},
	number = {401},
	journal = {Journal of the American Statistical Association},
	publisher = {[American Statistical Association, Taylor \& Francis, Ltd.]},
	author = {Ogata, Yosihiko},
	year = {1988},
	pages = {9--27},
}

@article{reynaud-bouret_non_2007,
	title = {Some non asymptotic tail estimates for {Hawkes} processes},
	volume = {13},
	issn = {1370-1444},
	doi = {10.36045/bbms/1170347811},
	number = {5},
	journal = {Bulletin of the Belgian Mathematical Society - Simon Stevin},
	publisher = {The Belgian Mathematical Society},
	author = {Reynaud-Bouret, Patricia and Roy, Emmanuel},
	month = jan,
	year = {2007},
	pages = {883--896},
}

@article{roueff_locally_2016,
	title = {Locally stationary {Hawkes} processes},
	volume = {126},
	issn = {0304-4149,1879-209X},
	doi = {10.1016/j.spa.2015.12.003},
	number = {6},
	journal = {Stochastic Processes and their Applications},
	author = {Roueff, François and von Sachs, Rainer and Sansonnet, Laure},
	year = {2016},
	mrnumber = {3483734},
	pages = {1710--1743},
}

@article{graham_regenerative_2021,
	title = {Regenerative properties of the linear {Hawkes} process with unbounded memory},
	volume = {31},
	doi = {10.1214/21-AAP1664},
	journal = {The Annals of Applied Probability},
	author = {Graham, Carl},
	month = dec,
	year = {2021},
}

@article{wang_total_2014,
	title = {On total progeny of multitype {Galton}-{Watson} process and the first passage time of random walk on lattice},
	volume = {30},
	issn = {1439-7617},
	doi = {10.1007/s10114-014-3650-1},
	language = {en},
	number = {12},
	journal = {Acta Mathematica Sinica, English Series},
	author = {Wang, Hua Ming},
	month = dec,
	year = {2014},
	pages = {2161--2172},
}

@misc{tchouanti_separation_2024,
	title = {Separation rates for the detection of synchronization of interacting point processes in a mean field frame. {Application} to neuroscience},
	doi = {10.48550/arXiv.2402.01919},
	publisher = {arXiv},
	author = {Tchouanti, Josué and Löcherbach, Eva and Reynaud-Bouret, Patricia and Tanré, Etienne},
	month = jul,
	year = {2024},
	note = {arXiv:2402.01919 [math]},
}

@article{leblanc_exponential_2024,
	title = {Exponential moments for {Hawkes} processes under minimal assumptions},
	volume = {29},
	issn = {1083-589X, 1083-589X},
	doi = {10.1214/24-ECP625},
	number = {none},
	journal = {Electronic Communications in Probability},
	publisher = {Institute of Mathematical Statistics and Bernoulli Society},
	author = {Leblanc, Théo},
	month = jan,
	year = {2024},
	pages = {1--11},
}

@book{athreya_branching_1972,
	address = {Berlin, Heidelberg},
	title = {Branching {Processes}},
	copyright = {http://www.springer.com/tdm},
	isbn = {978-3-642-65373-5 978-3-642-65371-1},
	doi = {10.1007/978-3-642-65371-1},
	publisher = {Springer},
	author = {Athreya, Krishna B. and Ney, Peter E.},
	year = {1972},
}

@article{lambert_reconstructing_2018,
	title = {Reconstructing the functional connectivity of multiple spike trains using {Hawkes} models},
	volume = {297},
	issn = {0165-0270},
	doi = {10.1016/j.jneumeth.2017.12.026},
	journal = {Journal of Neuroscience Methods},
	author = {Lambert, Régis C. and Tuleau-Malot, Christine and Bessaih, Thomas and Rivoirard, Vincent and Bouret, Yann and Leresche, Nathalie and Reynaud-Bouret, Patricia},
	month = mar,
	year = {2018},
	pages = {9--21},
}

@misc{boly_mixing_2023,
	title = {Mixing properties for multivariate {Hawkes} processes},
	doi = {10.48550/arXiv.2311.11730},
	publisher = {arXiv},
	author = {Boly, Ousmane and Cheysson, Felix and Nguyen, Thi Hien},
	month = nov,
	year = {2023},
	note = {arXiv:2311.11730 [math]},
}

@article{dwass_total_1969,
	title = {The {Total} {Progeny} in a {Branching} {Process} and a {Related} {Random} {Walk}},
	volume = {6},
	issn = {0021-9002},
	doi = {10.2307/3212112},
	number = {3},
	journal = {Journal of Applied Probability},
	publisher = {Applied Probability Trust},
	author = {Dwass, Meyer},
	year = {1969},
	pages = {682--686},
}

@book{harris_theory_1963,
	title = {The {Theory} of {Branching} {Processes}},
	language = {en},
	publisher = {Springer},
	author = {Harris, Theodore Edward},
	year = {1963},
}

@article{good_number_1949,
	title = {The number of individuals in a cascade process},
	volume = {45},
	issn = {1469-8064, 0305-0041},
	doi = {10.1017/S030500410002497X},
	language = {en},
	number = {3},
	journal = {Mathematical Proceedings of the Cambridge Philosophical Society},
	author = {Good, I. J.},
	month = jul,
	year = {1949},
	pages = {360--363},
}

@article{bondesson_class_2004,
	title = {A class of infinitely divisible distributions connected to branching processes and random walks},
	volume = {295},
	issn = {0022-247X},
	doi = {10.1016/j.jmaa.2004.03.018},
	number = {1},
	journal = {Journal of Mathematical Analysis and Applications},
	author = {Bondesson, Lennart and Steutel, Fred},
	month = jul,
	year = {2004},
	pages = {134--143},
}

@article{karim_compound_2025,
	title = {Compound multivariate {Hawkes} processes: {Large} deviations and rare event simulation},
	volume = {31},
	issn = {1350-7265},
	shorttitle = {Compound multivariate {Hawkes} processes},
	doi = {10.3150/24-BEJ1840},
	language = {en},
	number = {4},
	journal = {Bernoulli},
	publisher = {Bernoulli Society for Mathematical Statistics and Probability},
	author = {Karim, Raviar S. and Laeven, Roger J. A. and Mandjes, Michel},
	month = nov,
	year = {2025},
	pages = {3113--3138},
}

@article{bordenave_large_2007,
	title = {Large {Deviations} of {Poisson} {Cluster} {Processes}},
	volume = {23},
	issn = {1532-6349},
	doi = {10.1080/15326340701645959},
	number = {4},
	journal = {Stochastic Models},
	publisher = {Taylor \& Francis},
	author = {Bordenave, Charles and Torrisi, Giovanni Luca},
	month = nov,
	year = {2007},
	note = {\_eprint: https://doi.org/10.1080/15326340701645959},
	pages = {593--625},
}

@misc{karim_exact_2021,
	title = {Exact and {Asymptotic} {Analysis} of {General} {Multivariate} {Hawkes} {Processes} and {Induced} {Population} {Processes}},
	doi = {10.48550/arXiv.2106.03560},
	publisher = {arXiv},
	author = {Karim, Raviar and Laeven, Roger J. A. and Mandjes, Michel},
	month = jun,
	year = {2021},
	note = {arXiv:2106.03560 [math.PR]},
}

\end{document}